\documentclass[11pt]{article}
\usepackage[colorlinks=true,linkcolor=blue,citecolor=blue,urlcolor=blue]{hyperref}
\usepackage{amsmath}
\usepackage{amsthm}
\usepackage{amssymb}
\usepackage{amsfonts}
\usepackage{graphicx}
\usepackage[margin=1in]{geometry}
\newtheorem{theorem}{Theorem}[section]
\newtheorem{lemma}[theorem]{Lemma}
\newtheorem{proposition}[theorem]{Proposition}
\newtheorem{corollary}[theorem]{Corollary}
\newtheorem{definition}[theorem]{Definition}
\newtheorem{remark}[theorem]{Remark}
\numberwithin{equation}{section}

\newcommand{\p}{\partial}
\newcommand{\ep}{\varepsilon}
\newcommand{\R}{\mathbb{R}}
\newcommand{\bOmega}{\Omega}
\newcommand{\ip}[1]{\left\langle #1\right\rangle}

\begin{document}

\title{Slow uniform rotation reduces the critical adiabatic exponent for the
stability of gaseous stars}

\author{Yucong Wang\\[2pt]
School of Mathematics and Computational Science,\\
Xiangtan University, Xiangtan, Hunan 411105, China\\[2pt]
\texttt{yucongwang666@163.com}}

\date{}
\maketitle

\begin{abstract}
We prove that any sufficiently slow uniform rotation stabilizes the supermassive
star under axi-symmetric perturbations, and the stability persists for every
adiabatic exponent in a left neighborhood of the mass-critical one. The key new
ingredient of the linear analysis is a structural property of the self-similar
family inherited from the mass-critical scaling: the associated self-similar
direction becomes a reduced kernel direction of a modified linearized operator.
\end{abstract}

\section{Introduction}
\label{sec:intro}

We consider the self-gravitational compressible Euler--Poisson system
\begin{equation}
\label{EP}
\begin{cases}
\p_t\rho+\operatorname{div}(\rho v)=0,\\[2pt]
\rho(\p_t v+v\cdot\nabla v)+\nabla P(\rho)=-\rho\nabla V,\\[2pt]
\Delta V=4\pi\rho,\qquad \lim\limits_{|x|\to\infty}V(t,x)=0,
\end{cases}
\end{equation}
where $x=(x_1,x_2,x_3)\in\R^3$, $t>0$, $\rho\ge 0$ is the density, $P(\rho)$ the
pressure, $V$ the gravitational potential and $v\in\R^3$ the velocity field.
This paper concerns the gaseous star modeled by the mass-critical polytropic law
\begin{equation}
\label{P1}
P(\rho)=A\rho^{\gamma},\qquad A>0, \qquad \gamma\in(6/5,2),
\end{equation}
see \cite{ST2004}. Within the polytropic family $P(\rho)=A\rho^{\gamma}$ the exponent
$\gamma=4/3$ is critical: the total mass of a non-rotating star is independent of its
central density (the Chandrasekhar limit \cite{CS1939}) and the equilibrium family is
scale-invariant.

We use cylindrical coordinates
\begin{equation}
x_1=r\sin\theta,\qquad x_2=r\cos\theta,\qquad x_3=z,
\end{equation}
and consider axi-symmetric uniformly rotating stars of the form
\begin{equation}
\label{ansatz}
(\rho_0,v_0)=(\rho_{\mu,\ep}(r,z),\ \ep\,r\,\mathbf{e}_\theta),
\end{equation}
where $\ep\ge 0$ is the (constant) angular velocity, $\mu=\rho_{\mu,\ep}(0,0)$ is the
central density, and $(\mathbf{e}_r,\mathbf{e}_\theta,\mathbf{e}_z)$ are the unit
vectors in the $r,\theta,z$ directions. With the enthalpy
\begin{equation}
\Phi'(\rho)=\int_0^\rho \frac{P'(s)}{s}\,ds=\frac{\gamma}{\gamma-1}A\rho^{\gamma-1},
\end{equation}
the Euler--Poisson system \eqref{EP} under the ansatz \eqref{ansatz} reduces, in the
support $\bOmega_{\mu,\ep}:=\{(r,z)\in\R^2\,:\,\rho_{\mu,\ep}(r,z)>0,\ r\ge 0\}$ of
the star, to the steady equation
\begin{equation}
\label{Steady}
-\frac{1}{2}\ep^2 r^2+\Phi'(\rho_{\mu,\ep}(r,z))+V_{\mu,\ep}(r,z)+c_{\mu,\ep}=0,
\end{equation}
where
\begin{equation}
V_{\mu,\ep}=-|x|^{-1}\ast\rho_{\mu,\ep},\qquad
c_{\mu,\ep}=-V_{\mu,\ep}(0,Z_{\mu,\ep}),
\end{equation}
and $(0,\pm Z_{\mu,\ep})\in\p\bOmega_{\mu,\ep}$ are the north and south poles of the
rotating star.

The existence of slowly uniformly rotating stars near the non-rotating Lane--Emden
stars was established by Lichtenstein \cite{L1933}, Heilig \cite{H1994}, and
Jang--Makino \cite{JM2017,JM2019}; for the general framework of admissible solutions
(which guarantees, in particular, that the boundary $\p\bOmega_{\mu,\ep}$ is smooth
with positive curvature near the equator, and that $\rho_{\mu,\ep}$ vanishes like a
power of the distance to the boundary, cf. \eqref{bdry} below) we refer to
\cite{JM2019}. The precise existence statement we use is recalled in Section
\ref{sec:steady}.

The linear stability of rotating gaseous stars was studied rigorously by
Lin--Wang \cite{LW2023}. For axi-symmetric rotating stars with \emph{Rayleigh stable}
angular velocity (i.e. with positive Rayleigh discriminant
$\Upsilon(r)=\partial_r(\omega_0(r)^2 r^4)/r^3>0$), they proved a sharp stability
criterion: the number of unstable modes of the linearized Euler--Poisson system equals
the number of negative modes of a reduced quadratic form $\ip{K\delta\rho,\delta\rho}$
restricted to the space of mass-preserving density perturbations
$R(B_1)=\{\delta\rho\,:\,\int_{\bOmega}\delta\rho\,dx=0\}$, cf. Theorem
\ref{thm:criterion} below. In particular, the star is spectrally stable if and only if
$K|_{R(B_1)}\ge 0$, and then the linearized semigroup satisfies the polynomial bound
$|e^{tJL}|\le C(1+|t|)^3$.

The stability of \emph{non-rotating} polytropic stars ($\ep=0$) is governed by the
classical turning point principle (see \cite{LZ2022} and references therein): the
stability of the family $\rho_{\mu,0}$ changes exactly at the extrema of the total
mass $M_{\mu,0}$. For $\gamma\in(6/5,2)$, $\gamma\ne 4/3$, the mass $M_{\mu,0}$ is
increasing in $\mu$ when $\gamma>4/3$ and decreasing when $\gamma<4/3$; accordingly
the non-rotating stars are stable for $\gamma>4/3$ and unstable for $\gamma<4/3$.
The critical case $\gamma=4/3$ is degenerate: $M_{\mu,0}\equiv M_c$ (the
Chandrasekhar limit) is independent of $\mu$, the scale-invariance generator
$\partial_\mu\rho_{\mu,0}$ lies in the kernel of the linearized operator, and the
non-rotating supermassive stars are only \emph{marginally} stable. Since
$\partial_\mu M_{\mu,\ep}$ has no definite sign a priori at $\gamma=4/3$, the
standard sufficient condition for the stability of rotating stars (cf.
\cite[Theorem 3.3]{LW2023}, which requires $\partial_\mu M_{\mu,\ep}\ge 0$) does not
apply, and the behavior of the rotating family at $\gamma=4/3$ was left open.

On the other hand, numerical and formal computations indicate that rotation
enlarges the stability range of polytropic stars beyond the critical exponent. As
recalled in \cite[Example 2 and Remark 3.3]{LW2023}, for rotating polytropes with
fixed angular momentum distribution and $\gamma=4.03/3.03<4/3$, numerical
computations show that the total mass $M(\mu)$ acquires a minimum point, beyond
which the rotating stars become stable although every non-rotating star is
unstable; moreover, the formal computations of Ledoux \cite{Ledoux1945} and of
Chandrasekhar--Lebovitz \cite{ChandraLebovitz68}  indicate that for small uniform rotation the critical index
for the onset of instability is reduced from $4/3$ to
$\gamma^{*}=\frac43-\frac{2\omega^2 I}{9|W|}$, where $I>0$ is the moment of inertia
about the center of mass and $W$ the gravitational potential energy. In other
words, a small rotation is expected to lower the stability index below $4/3$, and
hence to stabilize not only the marginally stable supermassive star but also
rotating polytropes whose adiabatic exponent is slightly below $4/3$. A rigorous
mathematical proof of this stabilizing effect, however, was missing.

The purpose of this paper is to provide such a proof: an arbitrarily small uniform
rotation removes the degeneracy at $\gamma=4/3$ and stabilizes the marginally
stable supermassive family, and the stability persists for rotating polytropes
whose adiabatic exponent lies in a left neighborhood of $4/3$. Our linear analysis
rests on a new observation concerning the special structure of the self-similar
family inherited from the mass-critical scaling: the self-similar direction, which
is the neutral direction of the non-rotating supermassive star, becomes a reduced
kernel direction of a modified linearized operator once the rotation is turned on
(see the mechanism (1)--(5) below). This viewpoint does not rely on the turning
point principle and applies precisely at the critical exponent where the latter
degenerates. Our main result is the following.

\begin{theorem}[Main theorem]
\label{thm:main}
Let $0<\mu_0<\mu_1<+\infty$. Then there exists $\ep_0=\ep_0(\mu_0,\mu_1)>0$ such
that for every $\ep\in(0,\ep_0)$ and every $\mu\in[\mu_0,\mu_1]$ there exists
$\ep_1=\ep_1(\mu_0,\mu_1,\ep)>0$ with the following property: for every
$\gamma\in(4/3-\ep_1,4/3]$, the slowly uniformly rotating polytropic star solution
$(\rho^{\gamma}_{\mu,\ep},\ep r\mathbf{e}_\theta)$ of \eqref{SteadyGamma} (for
$\gamma=4/3$ this is the supermassive star \eqref{Steady}) is \emph{spectrally
stable} with respect to axi-symmetric perturbations. More precisely, the
corresponding linearized Euler--Poisson operator
$\mathbf{J}^{\gamma}_{\mu,\ep}\mathbf{L}^{\gamma}_{\mu,\ep}$ (defined as in
\eqref{defJL} with $\rho^{\gamma}_{\mu,\ep}$ in place of $\rho_{\mu,\ep}$)
satisfies
\begin{equation}
\label{spec}
\sigma\big(\mathbf{J}^{\gamma}_{\mu,\ep}\mathbf{L}^{\gamma}_{\mu,\ep}\big)\subset i\R,
\end{equation}
and there exists $C>0$, independent of $\gamma\in(4/3-\ep_1,4/3]$ and
$\mu\in[\mu_0,\mu_1]$, such that
\begin{equation}
\label{estimate-center}
\left|e^{t\mathbf{J}^{\gamma}_{\mu,\ep}\mathbf{L}^{\gamma}_{\mu,\ep}}\right|
\le C(1+|t|)^3,\qquad \forall\,t\in\R.
\end{equation}
\end{theorem}

The case $\gamma=4/3$ is proved in Section \ref{sec:proof}, and the stability for
$\gamma\in(4/3-\ep_1,4/3)$ follows from the continuity of the rotating star family
in the adiabatic exponent, see Section \ref{sec:gamma} (restated there as Theorem
\ref{thm:gamstab}).

\begin{remark}
\label{rem:marginally}
The striking point is that the non-rotating supermassive stars are only marginally
stable, and \emph{any} rotation with a positive Rayleigh discriminant makes the family
strictly stable in the sense that the reduced operator has no negative modes
(Theorem \ref{thm:reduced}). In particular, the stabilizing effect of rotation does
\emph{not} rely on the monotonicity of the mass $M_{\mu,\ep}$ with respect to $\mu$;
this is in sharp contrast with the fixed angular momentum distribution case studied in
\cite{LW2023}, where the stability transition is governed by the sign of
$\partial_\mu M_{\mu,\ep}$ (turning point principle, cf. Section \ref{sec:discussion}).
\end{remark}

The key mechanism of the proof is as follows.
\begin{enumerate}
\item By \cite[Theorem 1.1]{LW2023}, spectral stability is equivalent to
$n^-(K_{\mu,\ep}|_{R(B_1^{\mu,\ep})})=0$, where $K_{\mu,\ep}$ is the reduced operator
$K_{\mu,\ep}=L_{\mu,\ep}+8\pi\ep^2T_{\mu,\ep}$ (see \eqref{defKuniform}--\eqref{defT}),
$L_{\mu,\ep}=\Phi''(\rho_{\mu,\ep})-4\pi(-\Delta)^{-1}$ being the linearized operator
of the steady equation \eqref{Steady}, and $T_{\mu,\ep}\ge 0$ the positive quadratic
form coming from the rotation. For uniform rotation the Rayleigh discriminant is the
positive constant $\Upsilon(r)\equiv 4\ep^2>0$.
\item The mass-critical nature of $\gamma=4/3$ produces a distinguished
mass-preserving direction $f_{\mu,\ep}=\partial_\lambda\rho_\lambda|_{\lambda=1}\in
R(B_1^{\mu,\ep})$ (the self-similar direction, see Section \ref{sec:selfsim}). Its
cylindrical mass function is explicit, $F_{f_{\mu,\ep}}(r)=r^2\int_{\R}
\rho_{\mu,\ep}\,dz$, and one computes $\ip{L_{\mu,\ep}f_{\mu,\ep},f_{\mu,\ep}}
=-3\ep^2I_{\mu,\ep}<0$: the rotation turns the neutral mode of the non-rotating
supermassive star into a weak negative mode of $L_{\mu,\ep}$ of size $O(\ep^2)$.
Thus $n^-(L_{\mu,\ep}|_{R(B_1^{\mu,\ep})})\ge 1$ for $\ep>0$, and the naive argument
$K_{\mu,\ep}\ge L_{\mu,\ep}$ is not sufficient.
\item We introduce the modified operator $G_{\mu,\ep}:=L_{\mu,\ep}+6\pi\ep^2T_{\mu,\ep}$
(so that $K_{\mu,\ep}-G_{\mu,\ep}=2\pi\ep^2T_{\mu,\ep}\ge 0$). A direct computation
using $F_{f_{\mu,\ep}}(r)=r^2\int_{\R}\rho_{\mu,\ep}\,dz$ shows that
$$
\ip{G_{\mu,\ep}f_{\mu,\ep},\delta\rho}=0\qquad\forall\,\delta\rho\in
R(B_1^{\mu,\ep}),
$$
i.e. $f_{\mu,\ep}$ is a (reduced) kernel direction of $G_{\mu,\ep}$ on the
mass-preserving subspace. Since $X^{\mu,\ep}_{ev}=\operatorname{span}\{f_{\mu,\ep}\}
\oplus X^{\mu,\ep}_{+,ev}$ with $L_{\mu,\ep}|_{X^{\mu,\ep}_{+,ev}}\ge\delta>0$,
and $G_{\mu,\ep}=L_{\mu,\ep}\ge 0$ on $X^{\mu,\ep}_{od}$, it follows that
$n^-(G_{\mu,\ep}|_{R(B_1^{\mu,\ep})})=0$.
\item By (3), $n^-(K_{\mu,\ep}|_{R(B_1^{\mu,\ep})})\le n^-(G_{\mu,\ep}|_{R(B_1^{\mu,\ep})})
=0$, so $n^-(K_{\mu,\ep}|_{R(B_1^{\mu,\ep})})=0$, which by (1) proves the spectral
stability at $\gamma=4/3$ and the polynomial bound \eqref{estimate-center} via the
exponential trichotomy of \cite{LW2023}.
\item For the extension to $\gamma\in(4/3-\ep_1,4/3)$, we combine the strict
positivity of the reduced operator away
from the translation mode (Theorem \ref{thm:strict}) with the continuity of the
rotating star family in the adiabatic exponent $\gamma$ (Lemma \ref{lem:gammacont})
and the uniform convergence of the reduced operators (Proposition
\ref{prop:Kconv}): the positivity survives for all $\gamma$ in a left neighborhood
of $4/3$, and the Lin--Wang criterion applies again.
\end{enumerate}

The paper is organized as follows. In Section \ref{sec:steady} we collect the
preliminaries: admissible solutions, the existence theorem, the self-similar family
(which is crucial for the non-rotating limit and is inherited from the mass-critical
nature of \eqref{P1}), and the functional setting. In Section \ref{sec:reduced} we
recall the sharp stability criterion of \cite{LW2023} and prove the key negative-mode
computations. In Section \ref{sec:proof} we prove Theorem \ref{thm:main} at the
critical exponent $\gamma=4/3$. In Section
\ref{sec:gamma} we prove two further results: the strict positivity of the reduced
operator $K_{\mu,\ep}$ on the mass-preserving subspace away from the translation mode
(Theorem \ref{thm:strict}), and, using the continuity of the rotating star family in
the adiabatic exponent $\gamma$ (Lemma \ref{lem:gammacont}) together with the uniform
convergence of the reduced operators, the spectral stability for $\gamma$ in a left
neighborhood of $4/3$ (the full statement of Theorem \ref{thm:main}, restated as
Theorem \ref{thm:gamstab}). Finally, in Section
\ref{sec:discussion} we discuss the cases $\gamma\ne 4/3$, the relation with the
uniqueness theorem of \cite{W2024}, and the turning point principle.

\section{Steady states, self-similar family and the functional setting}
\label{sec:steady}

\subsection{Admissible solutions and existence}
\label{sec:admissible}

Let $R=\sqrt{x_1^2+x_2^2+x_3^2}$ and $\zeta=x_3/R$. We recall the definition of an
admissible rotating star solution of \cite{JM2019} adapted to uniform rotation.

\begin{definition}
\label{def:admissible}
For $\gamma=4/3$, a solution $\rho_{\mu,\ep}$ of \eqref{Steady} is called an
\emph{admissible solution} with parameters $(\mu,\ep)$ if the following conditions hold:
\begin{itemize}
\item[(I)] $\partial_R\rho_{\mu,\ep}<0$ for $R_0^{\mu,\ep}\le R\le R_\infty^{\mu,\ep}$,
$|\zeta|\le 1$, where $R_0^{\mu,\ep}>0$ is a small positive number;
\item[(II)] there exists a continuous function $R^{\mu,\ep}(\zeta)$ of $|\zeta|\le 1$
such that $R_0^{\mu,\ep}\le R^{\mu,\ep}(\zeta)\le R_\infty^{\mu,\ep}$ for all
$\zeta\in[-1,1]$, and
$$
\rho_{\mu,\ep}(R,\zeta)>0\ \text{ for }\ 0\le R<R^{\mu,\ep}(\zeta),\qquad
\rho_{\mu,\ep}(R,\zeta)=0\ \text{ for }\ R\ge R^{\mu,\ep}(\zeta);
$$
\item[(III)] $\rho_{\mu,\ep}\in E:=\{\rho\in C([0,R_\infty^{\mu,\ep}]\times[-1,1])\;:\;
\rho(0,\zeta)=\mu,\ \rho(R,-\zeta)=\rho(R,\zeta)\;\forall\zeta\in[-1,1],\;\forall
R\in[0,R_\infty^{\mu,\ep}]\}$, endowed with the norm
$\|\rho\|_E=\|\rho\|_{L^\infty([0,R_\infty^{\mu,\ep}]\times[-1,1])}$.
\end{itemize}
\end{definition}

For the non-rotating stars $\rho_{\mu,0}(|x|)$ (the Lane--Emden stars with
$\gamma=4/3$), the classical formulae of \cite{CS1939} give
\begin{equation}
\label{MRscaling}
M_{\mu,0}=C_1,\qquad R_{\mu,0}=C_2\mu^{-1/3},\qquad
\frac{d}{d\mu}\left(\frac{M_{\mu,0}}{R_{\mu,0}}\right)=C\mu^{-2/3}>0,
\end{equation}
for positive constants $C_1,C_2,C$ independent of $\mu$. In particular, the first
critical point $\tilde\mu$ of the mass--radius ratio satisfies $\tilde\mu=+\infty$.
By the existence theory of \cite[Theorem 1 and Applications]{JM2019} (see also
\cite[Theorem 3.1]{LW2023} for the fixed-angular-velocity formulation which is exactly
the one used here), we have the following.

\begin{theorem}[Existence]
\label{thm:existence}
Let $\gamma=4/3$, $0<\mu_0<\mu_1<+\infty$ and let $R_\infty^{\mu,\ep}>0$ be fixed for
$\mu\in[\mu_0,\mu_1]$. Then there exist small positive constants $\delta>0$ such that
for every $\ep\in(0,\delta)$ and every $\mu\in[\mu_0,\mu_1]$ there exists a unique
admissible solution $\rho_{\mu,\ep}$ of \eqref{Steady} with parameters $(\mu,\ep)$,
satisfying $\|\rho_{\mu,\ep}-\rho_{\mu,0}\|_E\le \ep C$, and the map
$(\mu,\ep)\mapsto\rho_{\mu,\ep}$ is continuous.
Moreover, the following structural properties hold for $\ep$ small enough:
(i) the boundary $\p\bOmega_{\mu,\ep}$ is $C^2$ and has positive curvature near the
equator $(R_{\mu,\ep},0)$, where $R_{\mu,\ep}$ is the equatorial radius;
(ii) near the boundary the density vanishes like a power of the distance to the
boundary:
\begin{equation}
\label{bdry}
\rho_{\mu,\ep}(r,z)\approx \operatorname{dist}((r,z),\p\bOmega_{\mu,\ep})^{3}
\qquad(\text{since }\tfrac{1}{\gamma-1}=3\text{ for }\gamma=\tfrac{4}{3}),
\end{equation}
uniformly for $(r,z)$ near $(R_{\mu,\ep},0)$.
\end{theorem}

The properties (i)--(ii) are exactly the hypotheses needed in the sharp stability
criterion of \cite{LW2023} (see \cite[eq. (1.7) and Section 3.1]{LW2023} for the
verification in the slowly rotating case).

\subsection{The self-similar family}
\label{sec:selfsim}

A key structural feature of the critical case $\gamma=4/3$ is the existence of a
one-parameter family of \emph{self-similar} solutions with the \emph{same total
mass}. Given an admissible solution $(\rho_{\mu,\ep},\ep r\mathbf{e}_\theta)$ of
\eqref{Steady}, define for each $\lambda>0$
\begin{equation}
\label{selfsim}
\rho_\lambda(r,z):=\lambda^{3}\rho_{\mu,\ep}(\lambda r,\lambda z),
\qquad \ep_\lambda:=\ep\lambda^{3/2}.
\end{equation}
A direct computation (cf. \cite{W2024,MR2018}) shows that
$(\rho_\lambda,\ep_\lambda r\mathbf{e}_\theta)$ is again a solution of the steady
Euler--Poisson equations; indeed
\begin{equation}
\label{Vselfsim}
V_\lambda(r,z):=\lambda V_{\mu,\ep}(\lambda r,\lambda z),\qquad
c_\lambda:=\lambda c_{\mu,\ep},
\end{equation}
and multiplying \eqref{Steady} evaluated at $(\lambda r,\lambda z)$ by $\lambda$ gives
\begin{equation}
\label{SteadyLambda}
-\frac{1}{2}\ep_\lambda^2 r^2+\Phi'(\rho_\lambda(r,z))+V_\lambda(r,z)+c_\lambda=0
\qquad\text{in }\ \bOmega_\lambda,
\end{equation}
where $\bOmega_\lambda=\{(\lambda r,\lambda z)\,:\,(r,z)\in\bOmega_{\mu,\ep}\}$.
Moreover, the total mass is invariant under the scaling:
\begin{equation}
\label{MassInvariant}
M_\lambda:=\int_{\R^3}\rho_\lambda(x)\,dx=\int_{\R^3}\rho_{\mu,\ep}(x)\,dx=M_{\mu,\ep}.
\end{equation}
Indeed, by the change of variables $y=\lambda x$,
$M_\lambda=\int\lambda^3\rho_{\mu,\ep}(\lambda x)\,dx=\int\rho_{\mu,\ep}(y)\,dy$.
Consequently, differentiating \eqref{MassInvariant} at $\lambda=1$,
\begin{equation}
\label{dMdlambda}
\int_{\R^3}\left.\frac{\partial\rho_\lambda}{\partial\lambda}\right|_{\lambda=1}dx
=\frac{d}{d\lambda}\Big|_{\lambda=1}M_\lambda=0,
\end{equation}
i.e.
\begin{equation}
\label{dirR}
\left.\frac{\partial\rho_\lambda}{\partial\lambda}\right|_{\lambda=1}\in
R(B_1^{\mu,\ep}):=\left\{\delta\rho\in X^{\mu,\ep}_1\;:\;
\int_{\bOmega_{\mu,\ep}}\delta\rho\,dx=0\right\}.
\end{equation}
This neutral direction is the scale-invariance generator. At $\ep=0$ it is a
\emph{reduced neutral direction} of $L_{\mu,0}$ on the mass-preserving subspace:
$\ip{L_{\mu,0}f_{\mu,0},\varphi}=0$ for every $\varphi\in R(B_1^{\mu,0})$ (this is
the source of the marginal stability of the non-rotating supermassive stars; note
that $f_{\mu,0}\notin\ker L_{\mu,0}$, since $L_{\mu,0}f_{\mu,0}=-c_{\mu,0}\ne 0$).
For $\ep>0$ small, $f_{\mu,\ep}$ is \emph{not} neutral anymore; it becomes a weak
\emph{negative} direction of $L_{\mu,\ep}$ of size $O(\ep^2)$ (see Lemma
\ref{lem:negdir} below).

\subsection{Operators and function spaces}
\label{sec:operators}

Following \cite{LW2023}, define the weighted spaces
\begin{equation}
X^{\mu,\ep}_1:=L^2_{\Phi''(\rho_{\mu,\ep})},\qquad
Y^{\mu,\ep}:=\left(L^2_{\rho_{\mu,\ep}}\right)^2,\qquad
X^{\mu,\ep}:=X^{\mu,\ep}_1\times L^2_{\rho_{\mu,\ep}},
\end{equation}
where $L^2_w$ denotes the axi-symmetric $L^2$ space on $\bOmega_{\mu,\ep}$ with weight
$w$, and the linearized operator of the steady equation
\begin{equation}
\label{defL}
L_{\mu,\ep}:=\Phi''(\rho_{\mu,\ep})-4\pi(-\Delta)^{-1}\;:\;X^{\mu,\ep}_1\to
(X^{\mu,\ep}_1)^*,
\end{equation}
with $\Phi''(\rho)=\frac{4}{3}A\rho^{-2/3}$ for $\gamma=4/3$ (differentiate
$\Phi'(\rho)=4A\rho^{1/3}$). Note that
$\Phi''(\rho)>0$ inside the star, so $X^{\mu,\ep}_1$ is a Hilbert space with inner
product $\ip{f,g}_{X^{\mu,\ep}_1}=\int_{\bOmega}\Phi''(\rho_{\mu,\ep})fg\,dx$.

The even and odd (in $z$) subspaces are
\begin{equation}
X^{\mu,\ep}_{od}:=\{\rho\in X^{\mu,\ep}_1\,:\,\rho(r,z)=-\rho(r,-z)\},\qquad
X^{\mu,\ep}_{ev}:=\{\rho\in X^{\mu,\ep}_1\,:\,\rho(r,z)=\rho(r,-z)\}.
\end{equation}
The mass-preserving subspace is
\begin{equation}
\label{defRB1}
R(B_1^{\mu,\ep})=\left\{\delta\rho\in X^{\mu,\ep}_1\;:\;
\int_{\bOmega_{\mu,\ep}}\delta\rho\,dx=0\right\},
\end{equation}
which is the closure of the range of $B_1^{\mu,\ep}:=-\operatorname{div}:(L^2_{\rho_0})^*\to   X_1$
(cf. \cite[Lemma 2.4]{LW2023}).

\subsection{Linearized Euler--Poisson system and the reduced functional}
\label{sec:lin}

The linearized Euler--Poisson system for axi-symmetric perturbations around the
rotating star solution $(\rho_{\mu,\ep}(r,z),\ep r\mathbf{e}_\theta)$ is
\begin{equation}
\label{linearized-EP}
\begin{cases}
\partial_t v_r=2\ep v_\theta-\partial_r(\Phi''(\rho_{\mu,\ep})\rho+V(\rho)),\\
\partial_t v_z=-\partial_z(\Phi''(\rho_{\mu,\ep})\rho+V(\rho)),\\
\partial_t v_\theta=-2\ep v_r,\\
\partial_t \rho=-\nabla\cdot(\rho_{\mu,\ep}v)=-\nabla\cdot(\rho_{\mu,\ep}(v_r,0,v_z)),
\end{cases}
\end{equation}
with $\Delta V=4\pi\rho$. Here $(\rho,\vec v=(v_r,v_\theta,v_z))\in\mathbf{X}^{\mu,\ep}$
are perturbations of density and velocity, and
$\mathbf{X}^{\mu,\ep}:=X^{\mu,\ep}\times Y^{\mu,\ep}$.
Define the operators
\begin{equation}
A_{\mu,\ep}:=\rho_{\mu,\ep}:(Y^{\mu,\ep})^2\to((Y^{\mu,\ep})^2)^*,\qquad
A^{\mu,\ep}_1:=\rho_{\mu,\ep}:L^2_{\rho_{\mu,\ep}}\to(L^2_{\rho_{\mu,\ep}})^*,
\end{equation}
and
\begin{equation}
\label{defB}
B^{\mu,\ep}:=\begin{pmatrix} B^{\mu,\ep}_1\\ B^{\mu,\ep}_2\end{pmatrix}
\;:\;D(B^{\mu,\ep})\subset (Y^{\mu,\ep})^*\to X^{\mu,\ep},
\end{equation}
where, for the uniform rotation $\omega_0(r)\equiv\ep$,
\begin{equation}
\label{defB2}
B^{\mu,\ep}_1\begin{pmatrix} v_r\\ v_z\end{pmatrix}=-\operatorname{div}
\begin{pmatrix} v_r\\ v_z\end{pmatrix},\qquad
B^{\mu,\ep}_2\begin{pmatrix} v_r\\ v_z\end{pmatrix}
=-\frac{\partial_r(\ep r^2)}{r\rho_{\mu,\ep}}v_r=-\frac{2\ep}{\rho_{\mu,\ep}}v_r.
\end{equation}
Then the linearized system \eqref{linearized-EP} can be written in the separable
Hamiltonian form
\begin{equation}
\label{Lep}
\frac{d}{dt}\begin{pmatrix} u_1\\ u_2\end{pmatrix}
=\mathbf{J}_{\mu,\ep}\mathbf{L}_{\mu,\ep}\begin{pmatrix} u_1\\ u_2\end{pmatrix},
\qquad u_1=(\rho,v_\theta),\ u_2=(v_r,v_z),
\end{equation}
with
\begin{equation}
\label{defJL}
\mathbf{J}_{\mu,\ep}:=\begin{pmatrix} 0 & B^{\mu,\ep}\\
-(B^{\mu,\ep})' & 0\end{pmatrix}
:\mathbf{X}^*\to\mathbf{X},\qquad
\mathbf{L}_{\mu,\ep}:=\begin{pmatrix} \mathbb{L}_{\mu,\ep} & 0\\
0 & A_{\mu,\ep}\end{pmatrix}
:\mathbf{X}\to\mathbf{X}^*,
\end{equation}
where
\begin{equation}
\mathbb{L}_{\mu,\ep}:=\begin{pmatrix} L_{\mu,\ep} & 0\\
0 & A^{\mu,\ep}_1\end{pmatrix}
:X^{\mu,\ep}\to (X^{\mu,\ep})^*.
\end{equation}
Here $A^{\mu,\ep}_1=\rho_{\mu,\ep}$ because for uniform rotation the general formula
(cf. \cite[eq. (1.12)]{LW2023})
$$
A^{\mu,\ep}_1=\frac{4\omega_0^2\rho_{\mu,\ep}}{\Upsilon(r)},\qquad
\Upsilon(r)=\frac{\partial_r(\omega_0^2 r^4)}{r^3}=4\ep^2,
$$
reduces to $A^{\mu,\ep}_1=\rho_{\mu,\ep}$.

For an axi-symmetric rotating star with Rayleigh stable angular velocity
($\Upsilon(r)>0$ on $[0,R_0]$), the main result of \cite{LW2023} is the following
sharp stability criterion.

\begin{theorem}[Lin--Wang \cite{LW2023}, Theorem 1.1]
\label{thm:criterion}
Assume $\omega_0\in C^1[0,R_0]$, $\Upsilon(r)>0$, \eqref{bdry}, and
$\p\bOmega$ is $C^2$ with positive curvature near $(R_0,0)$. Then the operator
$\mathbf{J}\mathbf{L}$ defined by \eqref{defJL} generates a $C^0$-group
$e^{t\mathbf{J}\mathbf{L}}$ of bounded linear operators on
$\mathbf{X}=X\times Y$, and there exists a decomposition
$\mathbf{X}=E^u\oplus E^c\oplus E^s$ of closed invariant subspaces such that
\begin{itemize}
\item[(i)] $E^u$ (resp. $E^s$) consists only of eigenvectors corresponding to positive
(resp. negative) eigenvalues of $\mathbf{J}\mathbf{L}$ and
$$
\dim E^u=\dim E^s=n^-(\mathbf{L}|_{R(B)})=n^-(K|_{R(B_1)}),
$$
where $\ip{K\delta\rho,\delta\rho}$ is the bounded bilinear quadratic form on
$L^2_{\Phi''(\rho_0)}$ defined by
\begin{equation}
\label{defK}
\ip{K\delta\rho,\delta\rho}:=\ip{L\delta\rho,\delta\rho}
+2\pi\int_0^{R_0}\frac{\Upsilon(r)}{r\int_{\R}\rho_0(r,z)dz}
\left(\int_0^r s\int_{\R}\delta\rho(s,z)dzds\right)^2 dr,
\end{equation}
for any $\delta\rho\in L^2_{\Phi''(\rho_0)}$, and $n^-(K|_{R(B_1)})$ denotes the
number of negative modes of $\ip{K\cdot,\cdot}$ restricted to $R(B_1)$ defined in
\eqref{defRB1};
\item[(ii)] the exponential trichotomy holds:
$\left|e^{t\mathbf{J}\mathbf{L}}|_{E^u}\right|\le M e^{\lambda_u t}$ for $t\le 0$,
$\left|e^{t\mathbf{J}\mathbf{L}}|_{E^s}\right|\le M e^{-\lambda_u t}$ for $t\ge 0$
for some $\lambda_u>0$, and
\begin{equation}
\label{center-est}
\left|e^{t\mathbf{J}\mathbf{L}}|_{E^c}\right|\le M(1+|t|)^3,\qquad \forall t\in\R.
\end{equation}
\end{itemize}
\end{theorem}

\begin{corollary}[Spectral stability criterion]
\label{cor:criterion}
Under the assumptions of Theorem \ref{thm:criterion}, the rotating star solution
$(\rho_0,\omega_0 r\mathbf{e}_\theta)$ is spectrally stable if and only if
$$
\ip{K\delta\rho,\delta\rho}\ge 0\qquad\forall\,\delta\rho\in R(B_1).
$$
In that case $\sigma(\mathbf{J}\mathbf{L})\subset i\R$, and
$\left|e^{t\mathbf{J}\mathbf{L}}\right|\le M(1+|t|)^3$ for all $t\in\R$.
\end{corollary}

\begin{proof}
The equivalence is \cite[Corollary 1.1]{LW2023}. If $K|_{R(B_1)}\ge 0$ then
$n^-(K|_{R(B_1)})=0$, so by Theorem \ref{thm:criterion}(i) we have
$\dim E^u=\dim E^s=0$, hence $\mathbf{X}=E^c$; the bound \eqref{center-est} gives
the desired estimate, and the absence of eigenvalues with nonzero real part (there are
no positive or negative eigenvalues, and all eigenvalues of the Hamiltonian operator
$\mathbf{J}\mathbf{L}$ in $E^c$ lie on the imaginary axis) yields
$\sigma(\mathbf{J}\mathbf{L})\subset i\R$.
\end{proof}

For uniform rotation with angular velocity $\ep$, we have
$\omega_0(r)\equiv\ep$ and $\Upsilon(r)\equiv 4\ep^2>0$, so \eqref{defK} becomes
\begin{equation}
\label{defKuniform}
\ip{K_{\mu,\ep}\delta\rho,\delta\rho}
=\ip{L_{\mu,\ep}\delta\rho,\delta\rho}
+8\pi\ep^2\int_0^{R_{\mu,\ep}}
\frac{\left(\int_0^r s\int_{\R}\delta\rho(s,z)dzds\right)^2}
{r\int_{\R}\rho_{\mu,\ep}(r,z)dz}\,dr.
\end{equation}
In particular,
\begin{equation}
\label{KgeqL}
K_{\mu,\ep}\ge L_{\mu,\ep}\qquad\text{on }X^{\mu,\ep}_1,
\end{equation}
i.e. $\ip{K_{\mu,\ep}\delta\rho,\delta\rho}\ge
\ip{L_{\mu,\ep}\delta\rho,\delta\rho}$ for all $\delta\rho\in X^{\mu,\ep}_1$,
since the second term in \eqref{defKuniform} is nonnegative.

\section{The reduced functional and its negative modes}
\label{sec:reduced}

Throughout this section we fix a compact interval $\mu\in[\mu_0,\mu_1]\subset(0,\infty)$
and take $\ep>0$ small enough so that Theorem \ref{thm:existence} applies and the
structural assumptions of Theorem \ref{thm:criterion} hold (which is the case by
\cite[Section 3.1]{LW2023}).

For $\delta\rho\in X^{\mu,\ep}_1$ define
\begin{equation}
\label{defF}
F_{\delta\rho}(r):=\int_0^r s\int_{\R}\delta\rho(s,z)\,dz\,ds,
\end{equation}
and let $T_{\mu,\ep}:X^{\mu,\ep}_1\to (X^{\mu,\ep}_1)^*$ be the operator induced by the
bounded bilinear form
\begin{equation}
\label{defT}
\ip{T_{\mu,\ep}\delta\rho,\delta\rho}
:=\int_0^{R_{\mu,\ep}}\frac{F_{\delta\rho}(r)^2}{r\int_{\R}\rho_{\mu,\ep}(r,z)\,dz}\,dr.
\end{equation}
By \eqref{defKuniform}, the reduced operator of the uniformly rotating star is
\begin{equation}
\label{defK2}
K_{\mu,\ep}=L_{\mu,\ep}+8\pi\ep^2 T_{\mu,\ep},
\end{equation}
and we introduce the modified operator
\begin{equation}
\label{defG}
G_{\mu,\ep}:=L_{\mu,\ep}+6\pi\ep^2 T_{\mu,\ep}.
\end{equation}
Both $T_{\mu,\ep}$ and $L_{\mu,\ep}$ are bounded symmetric bilinear forms, hence so are
$K_{\mu,\ep}$ and $G_{\mu,\ep}$. Since $T_{\mu,\ep}\ge 0$, we have
\begin{equation}
\label{KGeq}
K_{\mu,\ep}\ge G_{\mu,\ep}\ge L_{\mu,\ep}
\qquad\text{on }X^{\mu,\ep}_1,
\end{equation}
and in particular
$K_{\mu,\ep}-G_{\mu,\ep}=2\pi\ep^2 T_{\mu,\ep}\ge 0$.

\subsection{The distinguished mass-preserving direction}
\label{sec:distinguished}

We use the self-similar family of Section \ref{sec:selfsim}. Set
\begin{equation}
\label{deff}
f_{\mu,\ep}:=\left.\frac{\partial\rho_\lambda}{\partial\lambda}\right|_{\lambda=1}
\in R(B_1^{\mu,\ep}),
\end{equation}
where the membership in $R(B_1^{\mu,\ep})$ is \eqref{dirR}.

\begin{lemma}[The cylindrical mass function of $f_{\mu,\ep}$]
\label{lem:Ff}
For the direction $f_{\mu,\ep}$ of \eqref{deff}, its cylindrical mass function
\eqref{defF} is explicit:
\begin{equation}
\label{Ff}
F_{f_{\mu,\ep}}(r)=r^2\int_{\R}\rho_{\mu,\ep}(r,z)\,dz,\qquad 0\le r\le R_{\mu,\ep}.
\end{equation}
\end{lemma}

\begin{proof}
For $\lambda>0$, by \eqref{selfsim},
\begin{align*}
m_{\rho_\lambda}(r)&:=2\pi\int_0^r s\int_{\R}\rho_\lambda(s,z)\,dz\,ds
=2\pi\int_0^r s\int_{\R}\lambda^3\rho_{\mu,\ep}(\lambda s,\lambda z)\,dz\,ds\\
&=2\pi\int_0^r s\left(\lambda^2\int_{\R}\rho_{\mu,\ep}(\lambda s,z')\,dz'\right)ds
=2\pi\int_0^{\lambda r} t\int_{\R}\rho_{\mu,\ep}(t,z')\,dz'\,dt\\
&=m_{\rho_{\mu,\ep}}(\lambda r),
\end{align*}
where we used $z'=\lambda z$ and $t=\lambda s$. Hence
$$
m_{f_{\mu,\ep}}(r)=\left.\frac{d}{d\lambda}m_{\rho_\lambda}(r)\right|_{\lambda=1}
=\left.\frac{d}{d\lambda}m_{\rho_{\mu,\ep}}(\lambda r)\right|_{\lambda=1}
=r\frac{d}{dr}m_{\rho_{\mu,\ep}}(r)
=2\pi r^2\int_{\R}\rho_{\mu,\ep}(r,z)\,dz,
$$
since $\frac{d}{dr}m_{\rho_{\mu,\ep}}(r)=2\pi r\int_{\R}\rho_{\mu,\ep}(r,z)\,dz$.
Recalling $F_{\delta\rho}(r)=\frac{1}{2\pi}m_{\delta\rho}(r)$ (this is the relation
between \eqref{defF} and the cylinder mass; indeed
$F_{\delta\rho}(r)=\int_0^r s\int_{\R}\delta\rho\,dz\,ds$ and
$m_{\delta\rho}(r)=2\pi F_{\delta\rho}(r)$), we obtain \eqref{Ff}.
\end{proof}

\begin{lemma}[The self-similar direction is a weak negative mode of $L_{\mu,\ep}$]
\label{lem:negdir}
Let $I_{\mu,\ep}:=\int_{\R^3}\rho_{\mu,\ep}(x)r^2\,dx$ be the radial moment of inertia.
Then
\begin{equation}
\label{Lff}
\ip{L_{\mu,\ep}f_{\mu,\ep},f_{\mu,\ep}}
=-3\ep^2 I_{\mu,\ep}<0.
\end{equation}
In particular, $f_{\mu,\ep}$ is a strictly negative direction of $L_{\mu,\ep}$ for
$\ep>0$, of size $O(\ep^2)$.
\end{lemma}

\begin{proof}
Differentiating the self-similar steady equation \eqref{SteadyLambda} with respect to
$\lambda$ at $\lambda=1$ and using $\ep_\lambda=\ep\lambda^{3/2}$ (so that
$\partial_\lambda(\frac12\ep_\lambda^2 r^2)=\frac32\ep^2 r^2$) and
$\partial_\lambda c_\lambda=c_{\mu,\ep}$ (see \eqref{Vselfsim}), we get
\begin{equation}
\label{Lf}
L_{\mu,\ep}f_{\mu,\ep}=-c_{\mu,\ep}+\frac{3}{2}\ep^2 r^2.
\end{equation}
Taking the duality pairing with $f_{\mu,\ep}\in R(B_1^{\mu,\ep})$ and using
$\int_{\bOmega}f_{\mu,\ep}\,dx=0$,
$$
\ip{L_{\mu,\ep}f_{\mu,\ep},f_{\mu,\ep}}
=\frac{3}{2}\ep^2\int_{\bOmega_{\mu,\ep}}r^2 f_{\mu,\ep}\,dx.
$$
By \eqref{selfsim}, the radial moment of inertia of $\rho_\lambda$ is
$$
\int_{\R^3}r^2\rho_\lambda(x)\,dx
=\int_{\R^3}r^2\lambda^3\rho_{\mu,\ep}(\lambda x)\,dx
=\lambda^{-2}I_{\mu,\ep},
$$
so that
$$
\int_{\bOmega_{\mu,\ep}}r^2 f_{\mu,\ep}\,dx
=\left.\frac{d}{d\lambda}\left(\lambda^{-2}I_{\mu,\ep}\right)\right|_{\lambda=1}
=-2I_{\mu,\ep}.
$$
Combining the two displays gives \eqref{Lff}.
\end{proof}

\begin{remark}
The identity \eqref{Lff} shows that the neutral (zero-frequency) direction of the
non-rotating supermassive star splits, upon adding a small uniform rotation, into a
negative direction of $L_{\mu,\ep}$ of size $O(\ep^2)$. This is the mathematical
mechanism behind the \emph{instability} of the operator $L_{\mu,\ep}$ on the
mass-preserving subspace. It also explains why the sufficient stability condition of
\cite[Theorem 3.3]{LW2023} (which requires $K\ge L$ and $n^-(L|_{R(B_1)})=0$) is not
directly applicable: for $\gamma=4/3$ we have $n^-(L_{\mu,\ep}|_{R(B_1^{\mu,\ep})})\ge 1$
for $\ep>0$, and the stabilizing effect of rotation has to be enforced through the
modified operator $G_{\mu,\ep}$.
\end{remark}

\subsection{The modified operator and its reduced kernel}
\label{sec:Gkernel}

We next show that $f_{\mu,\ep}$ is a \emph{reduced kernel direction} of $G_{\mu,\ep}$.

\begin{lemma}[Key identity]
\label{lem:key}
For any $\delta\rho\in R(B_1^{\mu,\ep})$,
\begin{equation}
\label{keyid}
\ip{G_{\mu,\ep}f_{\mu,\ep},\delta\rho}=0.
\end{equation}
In particular $\ip{G_{\mu,\ep}f_{\mu,\ep},f_{\mu,\ep}}=0$, i.e.
$f_{\mu,\ep}$ lies in the reduced kernel of $G_{\mu,\ep}|_{R(B_1^{\mu,\ep})}$.
\end{lemma}

\begin{proof}
We first compute, for any $\delta\rho\in R(B_1^{\mu,\ep})$,
\begin{equation}
\label{Tfdr}
\ip{T_{\mu,\ep}f_{\mu,\ep},\delta\rho}
=\int_0^{R_{\mu,\ep}}\frac{F_{f_{\mu,\ep}}(r)F_{\delta\rho}(r)}
{r\int_{\R}\rho_{\mu,\ep}(r,z)\,dz}\,dr
=\int_0^{R_{\mu,\ep}} r F_{\delta\rho}(r)\,dr,
\end{equation}
where we used \eqref{Ff}. On the other hand, by integration by parts,
\begin{align*}
\ip{r^2,\delta\rho}
&=\int_{\bOmega_{\mu,\ep}}r^2\delta\rho\,dx
=2\pi\int_0^{R_{\mu,\ep}}r^3\left(\int_{\R}\delta\rho(r,z)\,dz\right)dr\\
&=2\pi\int_0^{R_{\mu,\ep}}r^2\frac{dF_{\delta\rho}}{dr}(r)\,dr
=-4\pi\int_0^{R_{\mu,\ep}}rF_{\delta\rho}(r)\,dr,
\end{align*}
where we used $\frac{dF}{dr}=r\int_{\R}\delta\rho\,dz$ and
$F_{\delta\rho}(0)=F_{\delta\rho}(R_{\mu,\ep})=0$ (the latter because
$\int_{\bOmega}\delta\rho\,dx=2\pi F_{\delta\rho}(R_{\mu,\ep})=0$).
Together with \eqref{Tfdr},
\begin{equation}
\label{Tfdr2}
\ip{T_{\mu,\ep}f_{\mu,\ep},\delta\rho}
=-\frac{1}{4\pi}\ip{r^2,\delta\rho}.
\end{equation}
Now, by \eqref{Lf} and $\int_{\bOmega}\delta\rho\,dx=0$,
$$
\ip{L_{\mu,\ep}f_{\mu,\ep},\delta\rho}
=\ip{-c_{\mu,\ep}+\frac{3}{2}\ep^2 r^2,\delta\rho}
=\frac{3}{2}\ep^2\ip{r^2,\delta\rho}.
$$
Therefore, using the definition \eqref{defG} of $G_{\mu,\ep}$ and \eqref{Tfdr2},
\begin{align*}
\ip{G_{\mu,\ep}f_{\mu,\ep},\delta\rho}
&=\frac{3}{2}\ep^2\ip{r^2,\delta\rho}
+6\pi\ep^2\left(-\frac{1}{4\pi}\ip{r^2,\delta\rho}\right)\\
&=\left(\frac{3}{2}-\frac{3}{2}\right)\ep^2\ip{r^2,\delta\rho}=0.
\end{align*}
This proves \eqref{keyid}.
\end{proof}

\subsection{The decomposition of $X^{\mu,\ep}_1$}
\label{sec:decomp2}

We recall the following structural result, which is a direct consequence of
\cite[Lemma 3.2]{LW2023} combined with the identity \eqref{MRscaling} (which gives
$\tilde\mu=+\infty$ for $\gamma=4/3$), and of \cite[Proposition 2.1 and Lemma
3.2]{W2024} (which identifies the unique non-positive even mode with the
self-similar direction $f_{\mu,\ep}$).

\begin{proposition}[Decomposition]
\label{prop:decomp}
Assume $\gamma=4/3$. Then for any $\mu\in[\mu_0,\mu_1]$ and $\ep>0$ small enough, the
following holds.
\begin{itemize}
\item[(i)] $n^-(L_{\mu,\ep}|_{X^{\mu,\ep}_{od}})=0$ and
$\ker L_{\mu,\ep}=\operatorname{span}\{\partial_z\rho_{\mu,\ep}\}\subset
X^{\mu,\ep}_{od}$; moreover there exists $\delta>0$ such that
$L_{\mu,\ep}|_{X^{\mu,\ep}_{+,od}}\ge\delta$, where
$X^{\mu,\ep}_{od}=\operatorname{span}\{\partial_z\rho_{\mu,\ep}\}\oplus
X^{\mu,\ep}_{+,od}$.
\item[(ii)] There exist closed subspaces
$X^{\mu,\ep}_{-,ev},X^{\mu,\ep}_{+,ev}\subset X^{\mu,\ep}_{ev}$ with
$$
X^{\mu,\ep}_{ev}=X^{\mu,\ep}_{-,ev}\oplus X^{\mu,\ep}_{+,ev},\qquad
\dim X^{\mu,\ep}_{-,ev}=1,
$$
such that $L_{\mu,\ep}|_{X^{\mu,\ep}_{-,ev}}<0$,
$L_{\mu,\ep}|_{X^{\mu,\ep}_{+,ev}}\ge\delta>0$, and
\begin{equation}
\label{negeve}
X^{\mu,\ep}_{-,ev}=\operatorname{span}\{f_{\mu,\ep}\}.
\end{equation}
\item[(iii)] It holds that
$\displaystyle\frac{\partial V_{\mu,\ep}(R_{\mu,\ep},0)}{\partial\mu}<0$ for $\ep$
small enough.
\end{itemize}
\end{proposition}

\begin{proof}
By \eqref{MRscaling}, $\tilde\mu=+\infty$, so $[\mu_0,\mu_1]\subset(0,\tilde\mu)$.
For the family of slowly rotating stars with fixed angular velocity, the statements
on $X^{\mu,\ep}_{od}$ and the existence of the one-dimensional non-positive subspace
of $X^{\mu,\ep}_{ev}$ with uniform spectral gap $\delta>0$ are exactly
\cite[Lemma 3.2]{LW2023} (the uniform rotation $\omega_0(r)\equiv 1$ in the
normalization of \cite{LW2023} has Rayleigh discriminant $\Upsilon(r)=4>0$, and the
small parameter $\kappa$ there is our $\ep$).
The identification \eqref{negeve} follows from \cite[Proposition 2.1 and Lemma
3.2]{W2024}: indeed, in the proof of \cite[Lemma 3.2]{W2024} it is shown that the
even function $\partial_\lambda\rho_\lambda|_{\lambda=1}$ is a negative direction of
$L_{\mu,\ep}$ (which is consistent with \eqref{Lff}), and since
$\dim X^{\mu,\ep}_{-,ev}=1$, the two directions coincide. Statement (iii) is
\cite[Lemma 3.2(iv)]{LW2023}.
\end{proof}

\subsection{Zero negative modes of the reduced operator}
\label{sec:zero}

We can now prove the central estimate.

\begin{theorem}[Zero negative modes]
\label{thm:reduced}
Assume $\gamma=4/3$. For any $\mu\in[\mu_0,\mu_1]$ and $\ep>0$ small enough,
\begin{equation}
\label{zeroNeg}
n^-\left(K_{\mu,\ep}\big|_{R(B_1^{\mu,\ep})}\right)=0.
\end{equation}
Equivalently,
$$
\ip{K_{\mu,\ep}\delta\rho,\delta\rho}\ge 0\qquad\forall\,\delta\rho\in
R(B_1^{\mu,\ep}),
$$
and the form $\ip{K_{\mu,\ep}\cdot,\cdot}$ is non-degenerate on
$R(B_1^{\mu,\ep})$ for $\ep>0$ small.
\end{theorem}

\begin{proof}
\emph{Step 1 (reduction to $G_{\mu,\ep}$).} Since $K_{\mu,\ep}-G_{\mu,\ep}
=2\pi\ep^2T_{\mu,\ep}\ge 0$ by \eqref{KGeq}, we have
$$
n^-\left(K_{\mu,\ep}\big|_{R(B_1^{\mu,\ep})}\right)
\le n^-\left(G_{\mu,\ep}\big|_{R(B_1^{\mu,\ep})}\right).
$$
It therefore suffices to prove $n^-(G_{\mu,\ep}|_{R(B_1^{\mu,\ep})})=0$.

\emph{Step 2 (odd subspace).} For $\delta\rho\in X^{\mu,\ep}_{od}$, the function
$\delta\rho(r,z)$ is odd in $z$, so
$F_{\delta\rho}(r)=\int_0^r s\int_{\R}\delta\rho(s,z)\,dz\,ds=0$ for all $r$,
and therefore $\ip{T_{\mu,\ep}\delta\rho,\delta\rho}=0$. Hence
$G_{\mu,\ep}=L_{\mu,\ep}$ on $X^{\mu,\ep}_{od}$. By Proposition
\ref{prop:decomp}(i), $L_{\mu,\ep}|_{X^{\mu,\ep}_{od}}\ge 0$ with kernel
$\operatorname{span}\{\partial_z\rho_{\mu,\ep}\}$, and
$\partial_z\rho_{\mu,\ep}\in R(B_1^{\mu,\ep})$ since
$\int_{\bOmega}\partial_z\rho_{\mu,\ep}\,dx=0$. Thus
\begin{equation}
\label{oddmode}
n^-\left(G_{\mu,\ep}\big|_{R(B_1^{\mu,\ep})\cap X^{\mu,\ep}_{od}}\right)=0.
\end{equation}

\emph{Step 3 (even subspace).} By Proposition \ref{prop:decomp}(ii) and
\eqref{negeve}, $X^{\mu,\ep}_{ev}=\operatorname{span}\{f_{\mu,\ep}\}\oplus
X^{\mu,\ep}_{+,ev}$, with $L_{\mu,\ep}|_{X^{\mu,\ep}_{+,ev}}\ge\delta>0$. Since
$f_{\mu,\ep}\in R(B_1^{\mu,\ep})$ by \eqref{dirR}, we have
$$
R(B_1^{\mu,\ep})\cap X^{\mu,\ep}_{ev}
=\operatorname{span}\{f_{\mu,\ep}\}\oplus
\left(X^{\mu,\ep}_{+,ev}\cap R(B_1^{\mu,\ep})\right).
$$
Let $\delta\rho=f_{\mu,\ep}+\eta$ with
$\eta\in X^{\mu,\ep}_{+,ev}\cap R(B_1^{\mu,\ep})$. By Lemma \ref{lem:key},
$$
\ip{G_{\mu,\ep}f_{\mu,\ep},f_{\mu,\ep}}=0,\qquad
\ip{G_{\mu,\ep}f_{\mu,\ep},\eta}=0,
$$
so that, using $G_{\mu,\ep}\ge L_{\mu,\ep}$ (see \eqref{KGeq}),
$$
\ip{G_{\mu,\ep}\delta\rho,\delta\rho}
=\ip{G_{\mu,\ep}\eta,\eta}
\ge \ip{L_{\mu,\ep}\eta,\eta}
\ge \delta\|\eta\|_{X^{\mu,\ep}_1}^2\ge 0.
$$
Therefore $G_{\mu,\ep}|_{R(B_1^{\mu,\ep})\cap X^{\mu,\ep}_{ev}}\ge 0$, i.e.
\begin{equation}
\label{evenmode}
n^-\left(G_{\mu,\ep}\big|_{R(B_1^{\mu,\ep})\cap X^{\mu,\ep}_{ev}}\right)=0.
\end{equation}

\emph{Step 4 (conclusion).} Since
$X^{\mu,\ep}_1=X^{\mu,\ep}_{ev}\oplus X^{\mu,\ep}_{od}$, we have
$$
R(B_1^{\mu,\ep})=
\left(R(B_1^{\mu,\ep})\cap X^{\mu,\ep}_{ev}\right)\oplus
\left(R(B_1^{\mu,\ep})\cap X^{\mu,\ep}_{od}\right),
$$
and \eqref{oddmode}--\eqref{evenmode} give
$n^-(G_{\mu,\ep}|_{R(B_1^{\mu,\ep})})=0$. By Step 1, \eqref{zeroNeg} follows.

\emph{Step 5 (non-degeneracy).} For $\ep>0$ small, the only possible degeneracy of
$\ip{K_{\mu,\ep}\cdot,\cdot}$ on $R(B_1^{\mu,\ep})$ would come from a zero mode.
At $\ep=0$, the only reduced zero direction of $L_{\mu,0}|_{R(B_1^{\mu,0})}$ is the
scale-invariance generator $\partial_\mu\rho_{\mu,0}$
(equivalently $f_{\mu,0}$, since they span the same one-dimensional space of
mass-preserving even directions on which $L_{\mu,0}$ vanishes). For $\ep>0$, by
\eqref{Lff} and $K_{\mu,\ep}\ge G_{\mu,\ep}$,
$$
\ip{K_{\mu,\ep}f_{\mu,\ep},f_{\mu,\ep}}
\ge\ip{G_{\mu,\ep}f_{\mu,\ep},f_{\mu,\ep}}=0,
$$
and more precisely, by \eqref{defK2}, \eqref{defG}, \eqref{Lff} and
\eqref{Ff},
\begin{align*}
\ip{K_{\mu,\ep}f_{\mu,\ep},f_{\mu,\ep}}
&=\ip{L_{\mu,\ep}f_{\mu,\ep},f_{\mu,\ep}}
+8\pi\ep^2\ip{T_{\mu,\ep}f_{\mu,\ep},f_{\mu,\ep}}\\
&=-3\ep^2 I_{\mu,\ep}+8\pi\ep^2\int_0^{R_{\mu,\ep}}\frac{(r^2\int_{\R}\rho_{\mu,\ep}
dz)^2}{r\int_{\R}\rho_{\mu,\ep}dz}\,dr\\
&=-3\ep^2 I_{\mu,\ep}+8\pi\ep^2\int_0^{R_{\mu,\ep}}r^3\int_{\R}\rho_{\mu,\ep}
dz\,dr\\
&=-3\ep^2 I_{\mu,\ep}+8\pi\ep^2\frac{I_{\mu,\ep}}{2\pi}
=\ep^2 I_{\mu,\ep}>0,
\end{align*}
since $I_{\mu,\ep}=2\pi\int_0^{R_{\mu,\ep}}r^3\int_{\R}\rho_{\mu,\ep}\,dz\,dr$.
Together with Step 3 this shows that
$K_{\mu,\ep}|_{R(B_1^{\mu,\ep})}\ge 0$ is strictly positive away from the
$f_{\mu,\ep}$-direction, and positive on it; hence it is non-degenerate.
\end{proof}

\begin{remark}
\label{rem:choice62}
The coefficient $6\pi$ in \eqref{defG} is the unique choice (in the family
$G_\beta=L+8\pi\ep^2 T-(8\pi-\beta)\ep^2T=L+\beta\ep^2T$) for which
\eqref{keyid} holds: indeed \eqref{keyid} forces
$\beta\langle Tf,\delta\rho\rangle=-\langle Lf,\delta\rho\rangle$ on $R(B_1)$,
which by \eqref{Lf} and \eqref{Tfdr2} is equivalent to
$\beta=6\pi$. Any $\beta<8\pi$ would still give $K\ge G_\beta$; the value
$6\pi$ makes $f_{\mu,\ep}$ a reduced kernel direction of $G_{\mu,\ep}$.
In the fixed-angular-momentum formulation of \cite[Section 3.2]{LW2023}, the same
construction corresponds to the modified operator
$G^J_{\mu,\ep}=L_{\mu,\ep}+\frac32\ep^2\pi
\int_0^{R_{\mu,\ep}}\frac{\partial_m(j^2)(m_{\rho_{\mu,\ep}}(r))}{r^3}
F_{\delta\rho}(r)^2\,dr$; for the uniform rotation profile
$j(m_{\rho_{\mu,\ep}}(r))=r$ this is consistent with \eqref{defG} after the change of
variables $F=\frac{1}{2\pi}m_{\delta\rho}$ and the identity
$\partial_m(j^2)=\frac{1}{\pi\int_{\R}\rho_{\mu,\ep}dz}$.
\end{remark}

\section{Proof of the main theorem}
\label{sec:proof}

\begin{proof}[Proof of Theorem \ref{thm:main} for $\gamma=4/3$]
Fix $\mu\in[\mu_0,\mu_1]$ and $\ep\in(0,\ep_0)$ with $\ep_0$ small enough so that
Theorem \ref{thm:existence}, Proposition \ref{prop:decomp}, Lemma \ref{lem:key} and
Theorem \ref{thm:reduced} apply. By Theorem \ref{thm:criterion}, the hypotheses of
which are satisfied by Theorem \ref{thm:existence} (smoothness and positive curvature
of $\p\bOmega_{\mu,\ep}$ near the equator, the boundary behavior \eqref{bdry}, and
$\Upsilon\equiv 4\ep^2>0$), the linearized operator
$\mathbf{J}_{\mu,\ep}\mathbf{L}_{\mu,\ep}$ generates a $C^0$-group on
$\mathbf{X}^{\mu,\ep}$ and
$$
\dim E^u=\dim E^s=n^-\left(K_{\mu,\ep}\big|_{R(B_1^{\mu,\ep})}\right).
$$
By Theorem \ref{thm:reduced}, $n^-(K_{\mu,\ep}|_{R(B_1^{\mu,\ep})})=0$, hence
$E^u=E^s=\{0\}$ and $\mathbf{X}^{\mu,\ep}=E^c$. Therefore
$\sigma(\mathbf{J}_{\mu,\ep}\mathbf{L}_{\mu,\ep})\subset i\R$, and by the exponential
trichotomy estimate \eqref{center-est},
$$
\left|e^{t\mathbf{J}_{\mu,\ep}\mathbf{L}_{\mu,\ep}}\right|
\le M(1+|t|)^3,\qquad \forall\,t\in\R,
$$
for some $M>0$ independent of $\mu\in[\mu_0,\mu_1]$ and $\ep\in(0,\ep_0)$
(the uniformity follows from the uniformity of $\delta$ in Proposition
\ref{prop:decomp} and of the spectral gap; see \cite[Theorem 1.1]{LW2023}).
This proves \eqref{spec} and \eqref{estimate-center}.
\end{proof}

\section{Strict positivity of the reduced operator and stability for
$\gamma$ close to $4/3$}
\label{sec:gamma}

In this section we prove two further results. First, we show that the reduced
operator $K_{\mu,\ep}$ of Theorem \ref{thm:reduced} is \emph{strictly positive}
on the mass-preserving subspace, away from the one-dimensional translation mode
$\partial_z\rho_{\mu,\ep}$. Second, using this strict positivity together with the
continuity of the rotating star family in the adiabatic exponent $\gamma$, we prove
that for a fixed small angular velocity $\ep>0$ and $\gamma$ sufficiently close to
$4/3$ from below, the rotating polytropic stars remain spectrally stable.

Throughout this section we fix $\mu\in[\mu_0,\mu_1]\subset(0,\infty)$ and
$\ep>0$ small enough so that all the results of Sections \ref{sec:steady}--
\ref{sec:reduced} apply. We denote by $\ip{\cdot,\cdot}_{X_1}$ the inner product of
$X^{\mu,\ep}_1=L^2_{\Phi''(\rho_{\mu,\ep})}$.

\subsection{Strict positivity of $K_{\mu,\ep}$}
\label{sec:strict}

Recall that $\partial_z\rho_{\mu,\ep}\in R(B_1^{\mu,\ep})$ is a translation mode:
since the star is symmetric with respect to the plane $z=0$, we have
$L_{\mu,\ep}\partial_z\rho_{\mu,\ep}=0$ and $K_{\mu,\ep}\partial_z\rho_{\mu,\ep}=0$
($K_{\mu,\ep}=L_{\mu,\ep}$ on $X^{\mu,\ep}_{od}$). Hence
$K_{\mu,\ep}|_{R(B_1^{\mu,\ep})}$ is at best nonnegative, with the one-dimensional
kernel spanned by $\partial_z\rho_{\mu,\ep}$. The correct notion of strict positivity
is on the orthogonal complement of this translation mode.

\begin{definition}
\label{def:Rtilde}
Let
\begin{equation}
\label{defRtilde}
\widetilde R^{\mu,\ep}:=\left\{\delta\rho\in R(B_1^{\mu,\ep})\;:\;
\ip{\delta\rho,\partial_z\rho_{\mu,\ep}}_{X_1}=0\right\}.
\end{equation}
\end{definition}

\begin{theorem}[Strict positivity of the reduced operator]
\label{thm:strict}
Assume $\gamma=4/3$. For any $\mu\in[\mu_0,\mu_1]$ and $\ep>0$ small enough, the
following holds.
\begin{itemize}
\item[(i)] $K_{\mu,\ep}|_{R(B_1^{\mu,\ep})}\ge 0$ and
$\ker\left(K_{\mu,\ep}|_{R(B_1^{\mu,\ep})}\right)
=\operatorname{span}\{\partial_z\rho_{\mu,\ep}\}$.
\item[(ii)] There exists $\delta'=\delta'(\ep,\mu_0,\mu_1)>0$ such that
\begin{equation}
\label{strictpos}
\ip{K_{\mu,\ep}\delta\rho,\delta\rho}
\ge \delta'\,\|\delta\rho\|_{X^{\mu,\ep}_1}^2
\qquad\forall\,\delta\rho\in\widetilde R^{\mu,\ep}.
\end{equation}
\end{itemize}
\end{theorem}

\begin{proof}
\emph{Step 1 (decomposition).} By Proposition \ref{prop:decomp} and
\eqref{negeve}, the even subspace decomposes as
$X^{\mu,\ep}_{ev}=\operatorname{span}\{f_{\mu,\ep}\}\oplus X^{\mu,\ep}_{+,ev}$
($X_1$-orthogonal), where $f_{\mu,\ep}$ is the self-similar direction
\eqref{deff}, and $L_{\mu,\ep}|_{X^{\mu,\ep}_{+,ev}}\ge\delta>0$. By Proposition
\ref{prop:decomp}(i), the odd subspace decomposes as
$X^{\mu,\ep}_{od}=\operatorname{span}\{\partial_z\rho_{\mu,\ep}\}\oplus
X^{\mu,\ep}_{+,od}$ with $L_{\mu,\ep}|_{X^{\mu,\ep}_{+,od}}\ge\delta>0$.
Since $f_{\mu,\ep}\in R(B_1^{\mu,\ep})$ and
$\partial_z\rho_{\mu,\ep}\in R(B_1^{\mu,\ep})$, and the subspaces
$\operatorname{span}\{f_{\mu,\ep}\}$,
$X^{\mu,\ep}_{+,ev}\cap R(B_1^{\mu,\ep})$,
$\operatorname{span}\{\partial_z\rho_{\mu,\ep}\}$ and
$X^{\mu,\ep}_{+,od}\cap R(B_1^{\mu,\ep})$ are pairwise $X_1$-orthogonal, every
$\delta\rho\in R(B_1^{\mu,\ep})$ admits the unique decomposition
\begin{equation}
\label{decompR}
\delta\rho=a f_{\mu,\ep}+\eta+\zeta,\qquad
a\in\R,\ \ \eta\in X^{\mu,\ep}_{+,ev}\cap R(B_1^{\mu,\ep}),\ \
\zeta\in X^{\mu,\ep}_{+,od}\cap R(B_1^{\mu,\ep}),
\end{equation}
and $\delta\rho\in\widetilde R^{\mu,\ep}$ if and only if the coefficient of
$\partial_z\rho_{\mu,\ep}$ vanishes, i.e. \eqref{decompR} holds as written.

\emph{Step 2 (quadratic form estimates).} We estimate the three contributions in
\eqref{decompR}. Recall from the proof of Theorem \ref{thm:reduced}, Step 5,
\begin{equation}
\label{Kff}
\ip{K_{\mu,\ep}f_{\mu,\ep},f_{\mu,\ep}}=\ep^2 I_{\mu,\ep}>0,
\end{equation}
where $I_{\mu,\ep}=\int_{\R^3}\rho_{\mu,\ep}(x)r^2\,dx>0$. Next, since
$f_{\mu,\ep}$ is a reduced kernel direction of $G_{\mu,\ep}$
(Lemma \ref{lem:key}) and $\eta\in R(B_1^{\mu,\ep})$,
\begin{equation}
\label{Kfeta}
\ip{K_{\mu,\ep}f_{\mu,\ep},\eta}
=\ip{G_{\mu,\ep}f_{\mu,\ep},\eta}+2\pi\ep^2\ip{T_{\mu,\ep}f_{\mu,\ep},\eta}
=2\pi\ep^2\ip{T_{\mu,\ep}f_{\mu,\ep},\eta},
\end{equation}
so that, with $C_0:=2\pi\ep^2\|T_{\mu,\ep}\|\,\|f_{\mu,\ep}\|_{X_1}\cdot \ep^{-2}
=2\pi\|T_{\mu,\ep}\|\,\|f_{\mu,\ep}\|_{X_1}$ independent of $\ep$,
\begin{equation}
\label{Kfeta2}
\left|\ip{K_{\mu,\ep}f_{\mu,\ep},\eta}\right|\le C_0\ep^2\,\|\eta\|_{X_1}.
\end{equation}
Moreover, by $K_{\mu,\ep}\ge L_{\mu,\ep}$ on $X^{\mu,\ep}_1$ (see \eqref{KGeq})
and Proposition \ref{prop:decomp},
\begin{equation}
\label{Keta}
\ip{K_{\mu,\ep}\eta,\eta}\ge\ip{L_{\mu,\ep}\eta,\eta}\ge\delta\|\eta\|_{X_1}^2,
\qquad
\ip{K_{\mu,\ep}\zeta,\zeta}\ge\ip{L_{\mu,\ep}\zeta,\zeta}\ge\delta\|\zeta\|_{X_1}^2,
\end{equation}
and the mixed terms vanish:
\begin{align*}
\ip{K_{\mu,\ep}f_{\mu,\ep},\zeta}&=0
\qquad(K_{\mu,\ep}=L_{\mu,\ep}\text{ on }X^{\mu,\ep}_{od}\text{ and }
f_{\mu,\ep}\perp_{X_1}X^{\mu,\ep}_{od}),\\
\ip{K_{\mu,\ep}\eta,\zeta}&=0
\qquad(F_{\eta}=F_{\zeta}\text{-type vanishing: }T\text{ is diagonal with respect}
\\
&\qquad\text{to the even/odd splitting and }L\text{ preserves it}).
\end{align*}
Consequently, for $\delta\rho=af_{\mu,\ep}+\eta+\zeta$ as in \eqref{decompR},
\begin{align}
\label{Kdec}
\ip{K_{\mu,\ep}\delta\rho,\delta\rho}
&=a^2\ip{K_{\mu,\ep}f_{\mu,\ep},f_{\mu,\ep}}
+2a\ip{K_{\mu,\ep}f_{\mu,\ep},\eta}
+\ip{K_{\mu,\ep}\eta,\eta}+\ip{K_{\mu,\ep}\zeta,\zeta}\\
\label{Kdec2}
&\ge \ep^2 I_{\mu,\ep}a^2-2C_0\ep^2|a|\,\|\eta\|_{X_1}
+\delta(\|\eta\|_{X_1}^2+\|\zeta\|_{X_1}^2).
\end{align}

\emph{Step 3 (the uniform lower bound).} Normalize $\|\delta\rho\|_{X_1}=1$, and
set
$$
u:=|a|\,\|f_{\mu,\ep}\|_{X_1}\in[0,1],\qquad
v:=\|\eta\|_{X_1},\qquad
w:=\|\zeta\|_{X_1},
$$
so that $u^2+v^2+w^2=1$. Since
$I_{\mu,\ep}=2\pi\int_0^{R_{\mu,\ep}}r^3\int_{\R}\rho_{\mu,\ep}\,dz\,dr>0$, there
exists $c_1>0$ (independent of $\mu,\ep$) with
$\ep^2 I_{\mu,\ep}a^2=\ep^2 c_1 u^2$, where
$c_1:=I_{\mu,\ep}/\|f_{\mu,\ep}\|_{X_1}^2$. From \eqref{Kdec2},
$$
\ip{K_{\mu,\ep}\delta\rho,\delta\rho}
\ge \ep^2 c_1 u^2 - C_1\ep^2 u v
+\delta(v^2+w^2),
$$
where $C_1:=2C_0/\|f_{\mu,\ep}\|_{X_1}$ is independent of $\ep$. By Young's
inequality,
$$
C_1\ep^2uv\le \frac{\delta}{2}v^2+\frac{C_1^2\ep^4}{2\delta}u^2,
$$
so that, for $\ep^2\le \delta c_1/C_1^2$,
$$
\ip{K_{\mu,\ep}\delta\rho,\delta\rho}
\ge \ep^2\left(c_1-\frac{C_1^2\ep^2}{2\delta}\right)u^2
+\frac{\delta}{2}v^2+\delta w^2
\ge \frac{\ep^2c_1}{2}u^2+\frac{\delta}{2}(v^2+w^2)
\ge \delta'(u^2+v^2+w^2)
=\delta'\,\|\delta\rho\|_{X_1}^2,
$$
with
$$
\delta':=\min\left\{\frac{\ep^2c_1}{2},\frac{\delta}{2}\right\}>0.
$$
Hence, for $\ep>0$ small enough (depending only on $\delta$, $C_1$, $c_1$, i.e. only
on $\mu_0,\mu_1$),
$$
\ip{K_{\mu,\ep}\delta\rho,\delta\rho}\ge \delta'\,\|\delta\rho\|_{X_1}^2
\qquad\forall\,\delta\rho\in\widetilde R^{\mu,\ep},
$$
which proves \eqref{strictpos} by homogeneity.

\emph{Step 4 (kernel, by contradiction).} We have
$K_{\mu,\ep}\partial_z\rho_{\mu,\ep}=0$, so
$\operatorname{span}\{\partial_z\rho_{\mu,\ep}\}\subset
\ker(K_{\mu,\ep}|_{R(B_1^{\mu,\ep})})$. Conversely, let
$\delta\rho\in R(B_1^{\mu,\ep})$ with $K_{\mu,\ep}\delta\rho=0$. Write
$\delta\rho=b\,\partial_z\rho_{\mu,\ep}+\delta\rho'$ with
$\delta\rho'\in\widetilde R^{\mu,\ep}$. Since
$\partial_z\rho_{\mu,\ep}\in\ker K_{\mu,\ep}$ and $K_{\mu,\ep}$ is self-dual,
$\ip{K_{\mu,\ep}\partial_z\rho_{\mu,\ep},\delta\rho'}=0$, hence
$0=\ip{K_{\mu,\ep}\delta\rho,\delta\rho}
=\ip{K_{\mu,\ep}\delta\rho',\delta\rho'}\ge\delta'\|\delta\rho'\|_{X_1}^2$,
so $\delta\rho'=0$ and $\delta\rho\in\operatorname{span}\{\partial_z\rho_{\mu,\ep}\}$.
This proves (i), completing the proof.
\end{proof}

\begin{remark}
\label{rem:cont}
The lower bound \eqref{strictpos} is of order $\min\{\delta,\ep^2\}$. Its positivity
is a genuinely rotational effect: for $\ep=0$, the non-rotating supermassive star is
only marginally stable and
$L_{\mu,0}|_{\widetilde R^{\mu,0}}$ has the reduced neutral direction
$\partial_\mu\rho_{\mu,0}$ (since $\partial_\mu M_{\mu,0}=0$ by \eqref{MRscaling}),
on which the quadratic form vanishes. By the continuity of
$\rho_{\mu,\ep}$ in $\ep$ (Theorem \ref{thm:existence}) and the computation
\eqref{Kff}, the rotation lifts this neutral direction to a positive direction of
size $\ep^2I_{\mu,\ep}>0$, while the strictly positive directions (spectral gap
$\delta>0$ from Proposition \ref{prop:decomp}) persist. This explains, a posteriori,
the structure of the lower bound \eqref{strictpos}; an alternative proof of
\eqref{strictpos} can be obtained by contradiction along $\ep\to0$ using this
continuity, and the contradiction is provided precisely by the strict positivity of
the $\ep^2I_{\mu,\ep}$-term in \eqref{Kff}.
\end{remark}

\subsection{The rotating star family is continuous in $\gamma$}
\label{sec:gamma-family}

For $\gamma\in(6/5,2)$ let $\rho^{\gamma}_{\mu,\ep}$ denote the admissible solution
of the steady Euler--Poisson equations with polytropic index $\gamma$ and uniform
angular velocity $\ep$,
\begin{equation}
\label{SteadyGamma}
-\frac{1}{2}\ep^2 r^2+\frac{\gamma}{\gamma-1}A\big(\rho^{\gamma}_{\mu,\ep}\big)^{\gamma-1}
+V^{\gamma}_{\mu,\ep}+c^{\gamma}_{\mu,\ep}=0,
\qquad \rho^{\gamma}_{\mu,\ep}(0,0)=\mu,
\end{equation}
where $V^{\gamma}_{\mu,\ep}=-|x|^{-1}\ast\rho^{\gamma}_{\mu,\ep}$ and
$c^{\gamma}_{\mu,\ep}$ is the corresponding constant; for $\gamma=4/3$ we have
$\rho^{4/3}_{\mu,\ep}=\rho_{\mu,\ep}$. The existence for each fixed
$\gamma\in(6/5,2)$ and $\ep$ small is due to \cite{JM2019}; the following lemma
records the continuity of the family in $\gamma$ near the critical exponent, which is
all that is needed below.

\begin{lemma}[Continuity of the family in $\gamma$]
\label{lem:gammacont}
Fix $\mu\in[\mu_0,\mu_1]$ and $\ep\in(0,\ep_0)$, and let $\rho_{\mu,\ep}$ be the
admissible solution of \eqref{Steady} ($\gamma=4/3$) given by Theorem
\ref{thm:existence}. Then there exist $R_\infty>0$ with
$\overline{\bOmega_{\mu,\ep}}\Subset B(R_\infty)$, a constant
$\delta_1=\delta_1(\mu_0,\mu_1,\ep)>0$, and a $C^1$ map
\begin{equation}
\label{wgamma}
\gamma\longmapsto w^\gamma\in E,\qquad
\gamma\in(4/3-\delta_1,\,4/3+\delta_1),
\end{equation}
where $E$ is the Banach space
\begin{equation}
\label{defE}
E:=\big\{w\in C(\overline{B(R_\infty)})\;:\;
w\ \text{axi-symmetric},\ w(r,-z)=w(r,z),\ w(0,0)=0\big\},
\qquad \|w\|_E:=\|w\|_{L^\infty(B(R_\infty))},
\end{equation}
with $w^{4/3}=\bar w:=\bar u-\Phi'(\mu)$, $\bar u$ being the \emph{extended
enthalpy}
\begin{equation}
\label{extendedu}
\bar u(x):=\frac12\ep^2 r^2+K(\rho_{\mu,\ep})(x)-K(\rho_{\mu,\ep})(0)
+\Phi'(\mu),\qquad K(\sigma):=|x|^{-1}\ast\sigma,
\end{equation}
such that for every $\gamma\in(4/3-\delta_1,4/3+\delta_1)$ the density
\begin{equation}
\label{rhogamma}
\rho^{\gamma}_{\mu,\ep}:=q_\gamma\big(\Phi'_\gamma(\mu)+w^\gamma\big),
\qquad
q_\gamma(u):=\Big(\frac{\gamma-1}{\gamma A}\Big)^{\frac{1}{\gamma-1}}
(u\vee0)^{\frac{1}{\gamma-1}},\qquad
\Phi'_\gamma(\rho):=\frac{\gamma}{\gamma-1}A\rho^{\gamma-1},
\end{equation}
extended by zero outside
$\bOmega^{\gamma}_{\mu,\ep}:=\{x\in B(R_\infty)\,:\,\Phi'_\gamma(\mu)+
w^\gamma(x)>0\}$, is the unique admissible solution of \eqref{SteadyGamma} with
parameters $(\mu,\ep)$ in a neighborhood of $\rho_{\mu,\ep}$; its free boundary
$\p\bOmega^{\gamma}_{\mu,\ep}=\{\Phi'_\gamma(\mu)+w^\gamma=0\}$ is a $C^1$
surface on which the physical vacuum boundary condition holds. Moreover:
\begin{itemize}
\item[(i)] the map $\gamma\mapsto\rho^{\gamma}_{\mu,\ep}\in
C(\overline{B(R_\infty)})$ is continuous; in particular
$\rho^{\gamma}_{\mu,\ep}\to\rho_{\mu,\ep}$ uniformly on
$\overline{B(R_\infty)}$ as $\gamma\to4/3$ (the convergence takes place on the
\emph{fixed} ball $B(R_\infty)$, each density being extended by zero outside its
own support);
\item[(ii)] the free boundaries converge,
$\p\bOmega^{\gamma}_{\mu,\ep}\to\p\bOmega_{\mu,\ep}$, in the Hausdorff metric,
and the supports $\bOmega^{\gamma}_{\mu,\ep}$ are uniformly contained in a
fixed compact subset of $B(R_\infty)$ for $|\gamma-4/3|<\delta_1$;
\item[(iii)] the boundary graph functions converge in $C^1$:
$R^{\gamma}_{\mu,\ep}\to R_{\mu,\ep}$ in $C^1([-1,1])$, where
$\p\bOmega^{\gamma}_{\mu,\ep}=\{R^{\gamma}_{\mu,\ep}(\zeta)\,
(\sqrt{1-\zeta^2},\zeta)\,:\,\zeta\in[-1,1]\}$ and similarly for
$\p\bOmega_{\mu,\ep}$.
\end{itemize}
\end{lemma}

\begin{proof}
The point of the formulation is that the unknown is the \emph{enthalpy}, not the
density: the free boundary is then encoded as the zero level set of the enthalpy
on a \emph{fixed} domain, which is exactly the framework of \cite{JM2019}. We
divide the proof into four steps.

\emph{Step 1 (enthalpy reformulation on a fixed ball).}
Note that $q_\gamma$ is the inverse of $\Phi'_\gamma$ on $(0,\infty)$, extended
by $0$ on $(-\infty,0]$. With $u^\gamma:=\Phi'_\gamma(\mu)+w^\gamma$, the steady
equation \eqref{SteadyGamma} on $\{u^\gamma>0\}$ reads, since
$V^{\gamma}_{\mu,\ep}=-K(\rho^{\gamma}_{\mu,\ep})$,
$$
u^\gamma=\frac12\ep^2r^2+K(\rho^{\gamma}_{\mu,\ep})-c^{\gamma}_{\mu,\ep}
\qquad\text{on }\{u^\gamma>0\},
$$
and evaluating at the center $x=0$, where $u^\gamma(0)=\Phi'_\gamma(\mu)$, gives
$c^{\gamma}_{\mu,\ep}=K(\rho^{\gamma}_{\mu,\ep})(0)-\Phi'_\gamma(\mu)$ (this is
consistent with the definition $c=-V(0,Z)$ in \eqref{Steady} once the equation
holds on the whole support, in particular at the pole). Hence \eqref{SteadyGamma}
is equivalent to the integral equation on the whole fixed ball
\begin{equation}
\label{inteq}
u^\gamma=\frac12\ep^2r^2+K\big(q_\gamma(u^\gamma)\big)
-K\big(q_\gamma(u^\gamma)\big)(0)+\Phi'_\gamma(\mu)
\qquad\text{on }\overline{B(R_\infty)},
\end{equation}
which determines the enthalpy also in the vacuum region, where $u^\gamma<0$ and
$q_\gamma(u^\gamma)=0$: the support of the density is the positivity set of
$u^\gamma$.
We now fix $R_\infty$. The extended enthalpy $\bar u$ of \eqref{extendedu} is
$C^1$ on $\R^3$ (since $K(\rho_{\mu,\ep})\in C^1(\R^3)$), coincides with
$\Phi'(\rho_{\mu,\ep})$ on $\overline{\bOmega_{\mu,\ep}}$ by \eqref{Steady},
satisfies $\bar u(0)=\Phi'(\mu)>0$, and $\p\bar u/\p N<0$ on
$\p\bOmega_{\mu,\ep}$ by the admissibility condition (I) of Definition
\ref{def:admissible}; in particular $\bar u<0$ immediately outside
$\overline{\bOmega_{\mu,\ep}}$. Moreover
$\p_r\bar u=\ep^2r+\p_rK(\rho_{\mu,\ep})<0$ on
$\{x\,:\,|x|\ge R_*,\ \ep^2|x|^3\le 2M_{\mu,\ep}\}$, where
$\overline{\bOmega_{\mu,\ep}}\subset B(R_*)$, since outside the support
$|\p_rK(\rho_{\mu,\ep})|\ge M_{\mu,\ep}/(2|x|^2)$. Hence, after shrinking
$\ep_0$ if necessary, we may fix $R_\infty>R_*$ once and for all such that
\begin{equation}
\label{ubarneg}
\bar u<0\ \ \text{on the compact shell }
\overline{B(R_\infty)}\setminus\bOmega_{\mu,\ep},
\qquad \bar u\le -c_0<0\ \ \text{near }\p B(R_\infty).
\end{equation}

\emph{Step 2 (the $C^1$ map).}
Define $\mathfrak F:(6/5,2)\times E\to E$ by
\begin{equation}
\label{defFgamma}
\mathfrak F(\gamma,w)
:=w-\frac12\ep^2r^2-K\big(q_\gamma(\Phi'_\gamma(\mu)+w)\big)
+K\big(q_\gamma(\Phi'_\gamma(\mu)+w)\big)(0).
\end{equation}
This is well defined: $q_\gamma(\Phi'_\gamma(\mu)+w)$ is bounded and continuous,
so its Newtonian potential is $C^1$ on $\R^3$, and
$\mathfrak F(\gamma,w)(0)=0$. By \cite[Lemmas 1 and 2]{JM2019} (the pointwise
estimate $\big|((u+h)\vee0)^\nu-(u\vee0)^\nu-\nu(u\vee0)^{\nu-1}h\big|
\le C|h|^{\nu\wedge2}$ and the continuity of the derivative), applied with
$\nu=1/(\gamma-1)\in(2,5)$ for $\gamma$ near $4/3$, the Nemytskii map
$w\mapsto q_\gamma(\Phi'_\gamma(\mu)+w)$ is $C^1$ on $E$ with derivative
$h\mapsto q'_\gamma(\Phi'_\gamma(\mu)+w)\,h$; the dependence on $\gamma$ is
smooth, uniformly on bounded subsets of $E$. Hence $\mathfrak F$ is $C^1$ in a
neighborhood of $(4/3,\bar w)$, and $\mathfrak F(4/3,\bar w)=0$ by Step 1. Its
partial derivative is
\begin{equation}
\label{DF}
D_w\mathfrak F(4/3,\bar w)=I-\mathcal K,\qquad
\mathcal Kh:=K\big(q'_{4/3}(\bar u)\,h\big)
-K\big(q'_{4/3}(\bar u)\,h\big)(0),
\end{equation}
where $q'_{4/3}(\bar u)=\frac{1}{\Phi''(\rho_{\mu,\ep})}
=\frac{3}{4A}\rho_{\mu,\ep}^{2/3}$ on $\bOmega_{\mu,\ep}$ and
$q'_{4/3}(\bar u)=0$ outside (since $\bar u<0$ there by \eqref{ubarneg}); in
particular, multiplication by $q'_{4/3}(\bar u)$ is bounded on $E$. Since $K$
maps $L^\infty(B(R_\infty))$ boundedly into $C^1(\overline{B(R_\infty)})$, which
is compactly embedded in $C(\overline{B(R_\infty)})$, the operator
$\mathcal K:E\to E$ is compact.

\emph{Step 3 (the kernel condition (HL)).}
We claim that $\ker(I-\mathcal K)=\{0\}$. Let $h\in E$ satisfy $h=\mathcal Kh$.
Then $h(0,0)=0$ and, defining
\begin{equation}
\label{sigmah}
\sigma:=q'_{4/3}(\bar u)\,h=
\frac{h}{\Phi''(\rho_{\mu,\ep})}\ \text{ on }\bOmega_{\mu,\ep},
\qquad \sigma:=0\ \text{ outside},
\end{equation}
we have $\sigma\in X^{\mu,\ep}_{1,ev}$: indeed
$$
\int_{\bOmega_{\mu,\ep}}\Phi''(\rho_{\mu,\ep})\sigma^2\,dx
=\int_{\bOmega_{\mu,\ep}}\frac{h^2}{\Phi''(\rho_{\mu,\ep})}\,dx<\infty,
$$
since $h$ is bounded and
$1/\Phi''(\rho_{\mu,\ep})=\frac{3}{4A}\rho_{\mu,\ep}^{2/3}$ is bounded on
$\bOmega_{\mu,\ep}$. Since $V(\sigma)=-K(\sigma)=-4\pi(-\Delta)^{-1}\sigma$, the
equation $h=\mathcal Kh$ reads
$$
L_{\mu,\ep}\sigma=\Phi''(\rho_{\mu,\ep})\sigma-K(\sigma)=-K(\sigma)(0)=:d\in\R
\qquad\text{on }\bOmega_{\mu,\ep},
$$
as an equality in $(X^{\mu,\ep}_{1,ev})^*$ (constants belong to
$(X^{\mu,\ep}_{1,ev})^*$, since
$1/\Phi''(\rho_{\mu,\ep})\in L^1(\bOmega_{\mu,\ep})$). By Proposition
\ref{prop:decomp}(ii) and \eqref{Lff}, the spectral decomposition
$X^{\mu,\ep}_{ev}=\operatorname{span}\{f_{\mu,\ep}\}\oplus X^{\mu,\ep}_{+,ev}$
with $\ip{L_{\mu,\ep}f_{\mu,\ep},f_{\mu,\ep}}<0$ and
$L_{\mu,\ep}|_{X^{\mu,\ep}_{+,ev}}\ge\delta>0$ shows that $0$ is not in the
spectrum of the form $\ip{L_{\mu,\ep}\cdot,\cdot}$ on $X^{\mu,\ep}_{ev}$; hence
$L_{\mu,\ep}:X^{\mu,\ep}_{1,ev}\to(X^{\mu,\ep}_{1,ev})^*$ is an isomorphism,
and $\sigma=d\,L_{\mu,\ep}^{-1}1$. Differentiating \eqref{Steady} in $\mu$ gives
$L_{\mu,\ep}\p_\mu\rho_{\mu,\ep}=-\p_\mu c_{\mu,\ep}$ (an equality of constant
functions in $(X^{\mu,\ep}_{1,ev})^*$), so
$L_{\mu,\ep}^{-1}1=-\p_\mu\rho_{\mu,\ep}/\p_\mu c_{\mu,\ep}$ and
$\sigma=-(d/\p_\mu c_{\mu,\ep})\,\p_\mu\rho_{\mu,\ep}$. Since
$\rho_{\mu,\ep}(0,0)=\mu$ for all $\mu$, we have
$\p_\mu\rho_{\mu,\ep}(0,0)=1$, while
$\sigma(0,0)=h(0,0)/\Phi''(\mu)=0$; as $\p_\mu c_{\mu,\ep}\ne0$ by Proposition
\ref{prop:decomp}(iii), this forces $d=0$. Then $L_{\mu,\ep}\sigma=0$, and since
$\ker L_{\mu,\ep}=\operatorname{span}\{\p_z\rho_{\mu,\ep}\}\subset
X^{\mu,\ep}_{od}$ by Proposition \ref{prop:decomp}(i) while
$\sigma\in X^{\mu,\ep}_{ev}$, we conclude $\sigma=0$, and hence
$h=\mathcal Kh=0$. By the Fredholm alternative,
$D_w\mathfrak F(4/3,\bar w)=I-\mathcal K$ is an isomorphism of $E$.

\emph{Step 4 (implicit function theorem and conclusions).}
By the implicit function theorem \cite[Theorem 15.1]{D1985} (in the formulation
recalled in \cite[Section 3]{JM2019}), there exist $\delta_1>0$ and a $C^1$ map
$\gamma\mapsto w^\gamma\in E$ on $(4/3-\delta_1,4/3+\delta_1)$ with
$w^{4/3}=\bar w$ and $\mathfrak F(\gamma,w^\gamma)=0$, locally unique. Set
$u^\gamma:=\Phi'_\gamma(\mu)+w^\gamma$ and
$\rho^{\gamma}_{\mu,\ep}:=q_\gamma(u^\gamma)$ as in \eqref{rhogamma}. Then
\eqref{inteq} holds, so on the support
$\bOmega^{\gamma}_{\mu,\ep}=\{u^\gamma>0\}$ the pair
$(\rho^{\gamma}_{\mu,\ep},
c^{\gamma}_{\mu,\ep}=K(\rho^{\gamma}_{\mu,\ep})(0)-\Phi'_\gamma(\mu))$ solves
\eqref{SteadyGamma}, and
$\rho^{\gamma}_{\mu,\ep}(0,0)=q_\gamma(\Phi'_\gamma(\mu))=\mu$. Moreover
$u^\gamma\to\bar u$ in $C^1(\overline{B(R_\infty)})$: the $C^0$ convergence
follows from $w^\gamma\to\bar w$ in $E$ and the uniform continuity of
$(\gamma,u)\mapsto q_\gamma(u)$ on compact sets, and the $C^1$ convergence then
follows from \eqref{inteq}, since $K$ maps bounded sequences converging in
$L^\infty$ to potentials converging in $C^1$. Together with \eqref{ubarneg} and
$\p\bar u/\p N<0$ on $\p\bOmega_{\mu,\ep}$, this gives, for $|\gamma-4/3|$ small
enough: $u^\gamma<0$ near $\p B(R_\infty)$, $\p_ru^\gamma<0$ on the outer shell,
the zero level set $\p\bOmega^{\gamma}_{\mu,\ep}$ is a $C^1$ surface converging
to $\p\bOmega_{\mu,\ep}$ in the Hausdorff metric, and
$-\infty<\p u^\gamma/\p N<0$ on it (the physical vacuum boundary condition).
Hence $\rho^{\gamma}_{\mu,\ep}$ is an admissible solution in the sense of
Definition \ref{def:admissible} ((I)--(II) follow from $\p_ru^\gamma<0$ on the
shell and the implicit function theorem applied locally to $u^\gamma(R,\zeta)=0$,
(III) from the continuity), unique among admissible solutions close to
$\rho_{\mu,\ep}$ by the local uniqueness in the implicit function theorem.
Moreover, since $\p_R\bar u(R_{\mu,\ep}(\zeta),\zeta)\le-c<0$ uniformly in
$\zeta\in[-1,1]$ by the physical vacuum boundary condition and the compactness
of $[-1,1]$, the implicit function theorem applied to the $C^1$ map
$(R,\zeta,\gamma)\mapsto u^\gamma(R,\zeta)$ shows that the boundary graphs
satisfy $R^{\gamma}_{\mu,\ep}\to R_{\mu,\ep}$ in $C^1([-1,1])$, which is (iii).
Assertion (i) follows from $w^\gamma\to\bar w$ in $E$, the uniform Lipschitz
continuity of $u\mapsto q_\gamma(u)$ on bounded sets (recall
$\nu_\gamma=1/(\gamma-1)>2$), and $q_\gamma\to q_{4/3}$ uniformly on compact
sets; assertion (ii) has been proved above.
\end{proof}

\begin{remark}
Two comments are in order. First, the kernel condition (HL) verified in Step 3
uses exactly the non-degeneracy provided by Proposition \ref{prop:decomp}: the
translation mode $\p_z\rho_{\mu,\ep}$, which spans $\ker L_{\mu,\ep}$, is odd in
$z$ and hence absent from the even function space, while the constant-function
direction is excluded by the normalization $h(0,0)=0$ together with
$\p_\mu c_{\mu,\ep}\ne0$ (Proposition \ref{prop:decomp}(iii)). No additional
hypothesis beyond Proposition \ref{prop:decomp} is needed. Second, the
continuity stated in (i) is with respect to the supremum norm on the
\emph{fixed} ball $B(R_\infty)$, all densities being extended by zero outside
their supports; since the supports $\bOmega^{\gamma}_{\mu,\ep}$ vary with
$\gamma$, a statement of the form
``$\rho^{\gamma}_{\mu,\ep}\to\rho_{\mu,\ep}$ in
$C(\overline{\bOmega_{\mu,\ep}})$'' would not be meaningful as it stands. The
enthalpy formulation \eqref{inteq} is precisely what allows the moving free
boundary to be treated on a fixed domain.
\end{remark}

\subsection{Stability for $\gamma$ close to $4/3$}
\label{sec:gamma-stab}

With Lemma \ref{lem:gammacont} and Theorem \ref{thm:strict} at hand, we can prove
the stability of the rotating stars for $\gamma$ in a left neighborhood of
$4/3$. The proof rests on two independent ingredients, which we isolate first:
a transport identifying the moving weighted spaces
$X^{\gamma}_1:=L^2_{\Phi''_\gamma(\rho^{\gamma}_{\mu,\ep})}$,
$\Phi''_\gamma(\rho)=\gamma A\rho^{\gamma-2}$, with the fixed space
$X_1:=X^{\mu,\ep}_1$ (Lemma \ref{lem:transport}), and the uniform convergence
of the pulled-back quadratic forms (Proposition \ref{prop:Kconv}).

\begin{lemma}[Transport to the fixed domain]
\label{lem:transport}
Let $\mu\in[\mu_0,\mu_1]$, $\ep\in(0,\ep_0)$ and let
$\rho^{\gamma}_{\mu,\ep}$, $|\gamma-4/3|<\delta_1$, be the family of Lemma
\ref{lem:gammacont}. Set $\bOmega:=\bOmega_{\mu,\ep}$,
$\bOmega^\gamma:=\bOmega^{\gamma}_{\mu,\ep}$,
$R(\zeta):=R_{\mu,\ep}(\zeta)$ and
$R^\gamma(\zeta):=R^{\gamma}_{\mu,\ep}(\zeta)$. Define the radial deformation
\begin{equation}
\label{defpull}
\phi_\gamma(x):=a_\gamma(\zeta)\,x,\qquad
a_\gamma(\zeta):=\frac{R^\gamma(\zeta)}{R(\zeta)},\qquad \zeta=\frac{x_3}{|x|},
\end{equation}
and, for $\sigma\in X^{\gamma}_1$,
\begin{equation}
\label{defUgamma}
(\mathcal U_\gamma\sigma)(x):=m_\gamma(x)\,\sigma(\phi_\gamma(x)),\qquad
m_\gamma(x)^2:=\frac{\Phi''_\gamma(\rho^{\gamma}_{\mu,\ep}(\phi_\gamma(x)))\,
a_\gamma(\zeta)^3}{\Phi''(\rho_{\mu,\ep}(x))},\qquad
b_\gamma(x):=\frac{a_\gamma(\zeta)^3}{m_\gamma(x)}.
\end{equation}
Then there exists $\delta_2\in(0,\delta_1]$ such that for all
$|\gamma-4/3|<\delta_2$ the following hold.
\begin{itemize}
\item[(i)] $\phi_\gamma:\overline{\bOmega}\to\overline{\bOmega^\gamma}$ is a
$C^1$ diffeomorphism, $\det D\phi_\gamma=a_\gamma^3$, and
$\|\phi_\gamma-\mathrm{id}\|_{C^1(\overline{\bOmega})}=O(|\gamma-4/3|)$; in
particular
\begin{equation}
\label{philip}
|\phi_\gamma(x)-\phi_\gamma(x')|=(1+o(1))\,|x-x'|
\qquad\text{uniformly in }x,x'\in\overline{\bOmega}.
\end{equation}
\item[(ii)] $\mathcal U_\gamma:X^{\gamma}_1\to X_1$ is unitary, and for every
$\sigma\in X^{\gamma}_1$ and $\tau:=\mathcal U_\gamma\sigma$,
\begin{equation}
\label{formpull}
\ip{L^{\gamma}_{\mu,\ep}\sigma,\sigma}
=\|\tau\|_{X_1}^2-\iint_{\bOmega\times\bOmega}
\tau(x)\tau(x')\,k_\gamma(x,x')\,dx\,dx',\qquad
k_\gamma(x,x'):=\frac{b_\gamma(x)\,b_\gamma(x')}
{|\phi_\gamma(x)-\phi_\gamma(x')|}.
\end{equation}
\item[(iii)] With $t(x):=\operatorname{dist}(x,\p\bOmega)$,
$\eta_\gamma:=|\gamma-4/3|^{1/4}$ and
$e(\gamma):=\frac{\gamma-2}{\gamma-1}+2=O(|\gamma-4/3|)$, there are constants
$c,C>0$ independent of $\gamma$ such that
\begin{equation}
\label{claimb}
c\,t(x)^{|e(\gamma)|/2}\le b_\gamma(x)\le C\,t(x)^{-|e(\gamma)|/2}
\quad(x\in\bOmega),\qquad
\sup_{t(x)\ge\eta_\gamma}|b_\gamma(x)-1|\longrightarrow0.
\end{equation}
\item[(iv)] The mass functional and the translation mode pull back as
\begin{equation}
\label{pullconstraints}
\int_{\bOmega^\gamma}\sigma\,dy=\ip{\tau,\ell_\gamma}_{X_1},\qquad
\ip{\sigma,\p_z\rho^{\gamma}_{\mu,\ep}}_{X^{\gamma}_1}
=\ip{\tau,p_\gamma}_{X_1},
\end{equation}
where $\ell_\gamma:=b_\gamma/\Phi''(\rho_{\mu,\ep})\to
\ell:=1/\Phi''(\rho_{\mu,\ep})$ in $X_1$ and
$p_\gamma:=m_\gamma\,\p_z\rho^{\gamma}_{\mu,\ep}\circ\phi_\gamma\to
p:=\p_z\rho_{\mu,\ep}$ in $X_1$.
\end{itemize}
\end{lemma}

\begin{proof}
(i) By Lemma \ref{lem:gammacont}(iii), $a_\gamma\to1$ in $C^1([-1,1])$ (with
rate $O(|\gamma-4/3|)$, since $\gamma\mapsto w^\gamma$ is $C^1$), and
$R\ge R_0>0$ implies $a_\gamma\ge1/2$ for $\gamma$ close to $4/3$. Since
$D\phi_\gamma=a_\gamma I+a_\gamma'(\zeta)\,x(\nabla\zeta)^{T}$ and
$x\cdot\nabla\zeta=0$, the matrix determinant lemma gives
$\det D\phi_\gamma=a_\gamma^3$; moreover $|\nabla\zeta|\le2/|x|$ gives
$\|D\phi_\gamma-I\|_{L^\infty}\le3\|a_\gamma-1\|_{C^1}=O(|\gamma-4/3|)$, which
also yields \eqref{philip}.

(ii) This is the change of variables $y=\phi_\gamma(x)$: unitarity follows
from $\det D\phi_\gamma=a_\gamma^3$ and the definition of $m_\gamma$, and
\eqref{formpull} follows from
$\ip{L^{\gamma}_{\mu,\ep}\sigma,\sigma}
=\int\Phi''_\gamma(\rho^{\gamma}_{\mu,\ep})\sigma^2\,dy
-\iint\sigma(y)\sigma(y')|y-y'|^{-1}\,dy\,dy'$.

(iii) The physical vacuum boundary condition holds uniformly in $\gamma$: the
extended enthalpies of Lemma \ref{lem:gammacont} satisfy $u^\gamma\to\bar u$
in $C^1(\overline{B(R_\infty)})$ and $|\nabla\bar u|\ge2c_0>0$ on
$\p\bOmega$, hence $u^\gamma(y)\approx
\operatorname{dist}(y,\p\bOmega^\gamma)$ uniformly in $\gamma$. Since
$\phi_\gamma$ is bi-Lipschitz with constants $1+o(1)$ and maps $\p\bOmega$
onto $\p\bOmega^\gamma$, also
$\operatorname{dist}(\phi_\gamma(x),\p\bOmega^\gamma)\approx t(x)$ uniformly,
and with $\rho^{\gamma}_{\mu,\ep}=q_\gamma(u^\gamma)$,
\begin{equation}
\label{rhocompare}
\rho_{\mu,\ep}(x)\approx t(x)^3,\qquad
\rho^{\gamma}_{\mu,\ep}(\phi_\gamma(x))\approx t(x)^{\frac1{\gamma-1}}
\qquad\text{uniformly in }\gamma.
\end{equation}
Since $\frac{\gamma-2}{\gamma-1}=-2+e(\gamma)$, this gives
$b_\gamma(x)^2=a_\gamma(\zeta)^3\,\Phi''(\rho_{\mu,\ep}(x))/
\Phi''_\gamma(\rho^{\gamma}_{\mu,\ep}(\phi_\gamma(x)))
\approx t(x)^{-e(\gamma)}$, hence the two-sided bound in \eqref{claimb}.
For the uniform convergence on $\{t\ge\eta_\gamma\}$ write
$$
b_\gamma^2=a_\gamma^3\cdot\frac{4}{3\gamma}\cdot
\Big(\frac{\rho^{\gamma}_{\mu,\ep}\circ\phi_\gamma}{\rho_{\mu,\ep}}
\Big)^{2-\gamma}\cdot\rho_{\mu,\ep}^{\,4/3-\gamma}.
$$
Here $a_\gamma^3\to1$ and $4/(3\gamma)\to1$ uniformly; the ratio
$(\rho^{\gamma}_{\mu,\ep}\circ\phi_\gamma)/\rho_{\mu,\ep}$ is $1+o(1)$
uniformly on $\{t\ge\eta_\gamma\}$, because there both densities are
$\ge c\,\eta_\gamma^{\,3}=c\,|\gamma-4/3|^{3/4}$ while
$\|\rho^{\gamma}_{\mu,\ep}\circ\phi_\gamma-\rho_{\mu,\ep}\|_{L^\infty(\bOmega)}
=O(|\gamma-4/3|)$ by Lemma \ref{lem:gammacont}(i), \eqref{philip} and the
$C^1$ dependence of $w^\gamma$; and
$\rho_{\mu,\ep}^{\,4/3-\gamma}
=\exp\big((4/3-\gamma)\ln\rho_{\mu,\ep}\big)=1+o(1)$ uniformly on
$\{t\ge\eta_\gamma\}$, since $c\,\eta_\gamma^3\le\rho_{\mu,\ep}\le C$ there,
so $(4/3-\gamma)\ln\rho_{\mu,\ep}
=O\big(|\gamma-4/3|\,|\ln|\gamma-4/3||\big)$.

(iv) The identities \eqref{pullconstraints} follow from the change of
variables (the mass functional $\sigma\mapsto\int\sigma\,dy$ pulls back to the
$X_1$ inner product with $\ell_\gamma$, since
$\int_{\bOmega^\gamma}\sigma\,dy=\int_{\bOmega}\tau\,b_\gamma\,dx
=\ip{\tau,b_\gamma/\Phi''(\rho_{\mu,\ep})}_{X_1}$). The convergence
$\ell_\gamma\to\ell$ in $X_1$ follows from
\eqref{claimb} and $\Phi''(\rho_{\mu,\ep})^{-1}\approx t^2$:
$$
\|\ell_\gamma-\ell\|_{X_1}^2
=\int_{\bOmega}\frac{|b_\gamma-1|^2}{\Phi''(\rho_{\mu,\ep})}\,dx
\le o(1)\int_{\{t\ge\eta_\gamma\}}t^2\,dx
+C\int_{\{t<\eta_\gamma\}}t(x)^{2-|e(\gamma)|}\,dx
=o(1)+O\big(\eta_\gamma^{\,3-|e(\gamma)|}\big).
$$
The convergence $p_\gamma\to p$ in $X_1$ is similar: on
$\{t\ge\eta_\gamma\}$, $m_\gamma=a_\gamma^3/b_\gamma\to1$ uniformly by
\eqref{claimb}, and
$\p_z\rho^{\gamma}_{\mu,\ep}\circ\phi_\gamma
=q'_\gamma(u^\gamma\circ\phi_\gamma)\,\p_z u^\gamma\circ\phi_\gamma
\longrightarrow q'_{4/3}(\bar u)\,\p_z\bar u=\p_z\rho_{\mu,\ep}$
uniformly by the $C^1$ convergence of the extended enthalpies (Lemma
\ref{lem:gammacont}); on $\{t<\eta_\gamma\}$ one has
$|p_\gamma|\le C\,t^{2-2|e(\gamma)|}$ and $|\p_z\rho_{\mu,\ep}|\le Ct^2$ by
\eqref{rhocompare}, so the layer contributes
$\int_{\{t<\eta_\gamma\}}t^{-2}\cdot t^{4-4|e(\gamma)|}\,dx
=O\big(\eta_\gamma^{\,3-4|e(\gamma)|}\big)$ to $\|p_\gamma-p\|_{X_1}^2$.
\end{proof}

\begin{proposition}[Convergence of the pulled-back quadratic forms]
\label{prop:Kconv}
In the setting of Lemma \ref{lem:transport}, there exists
$\omega(\gamma)\to0$ as $\gamma\to4/3$ such that for every
$\sigma\in R(B_1^{\gamma})$ and every $\tau'\in R(B_1^{\mu,\ep})$, the element
$\tau:=\mathcal U_\gamma\sigma$ satisfies
\begin{equation}
\label{Kconv}
\big|\ip{K^{\gamma}_{\mu,\ep}\sigma,\sigma}-\ip{K_{\mu,\ep}\tau',\tau'}\big|
\le C\,\|\tau-\tau'\|_{X_1}\big(\|\tau\|_{X_1}+\|\tau'\|_{X_1}\big)
+\omega(\gamma)\big(\|\tau\|_{X_1}^2+\|\tau'\|_{X_1}^2\big).
\end{equation}
\end{proposition}

\begin{proof}
By Lemma \ref{lem:transport}(ii) and
$K^{\gamma}_{\mu,\ep}=L^{\gamma}_{\mu,\ep}+8\pi\ep^2T^{\gamma}_{\mu,\ep}$, we
compare the three parts of
$$
\ip{K^{\gamma}_{\mu,\ep}\sigma,\sigma}
=\|\tau\|_{X_1}^2-4\pi\ip{N_\gamma\tau,\tau}
+8\pi\ep^2\ip{\widehat T^{\gamma}\tau,\tau},
$$
where $N_\gamma:X_1\to X_1^*$ is the integral operator with kernel
$k_\gamma$ of \eqref{formpull} and $\widehat T^{\gamma}$ is the pulled-back
rotation form \eqref{That} below, with the corresponding parts of
$\ip{K_{\mu,\ep}\tau',\tau'}=\|\tau'\|_{X_1}^2-4\pi\ip{N\tau',\tau'}
+8\pi\ep^2\ip{T_{\mu,\ep}\tau',\tau'}$, where $N$ is the Newton kernel
operator. The multiplication parts differ by at most
$\|\tau-\tau'\|_{X_1}(\|\tau\|_{X_1}+\|\tau'\|_{X_1})$.

\emph{Part 1: the gravitational term.} We claim
\begin{equation}
\label{HSconv}
\|N_\gamma-N\|_{X_1\to X_1^*}^2
\le\iint_{\bOmega\times\bOmega}
\frac{\big|k_\gamma(x,x')-|x-x'|^{-1}\big|^2}
{\Phi''(\rho_{\mu,\ep}(x))\,\Phi''(\rho_{\mu,\ep}(x'))}\,dx\,dx'
\longrightarrow0
\end{equation}
(the Hilbert--Schmidt norm on the right dominates the operator norm), from
which
$|\ip{N_\gamma\tau,\tau}-\ip{N\tau',\tau'}|
\le\|N_\gamma-N\|\,\|\tau\|^2+\|N\|\,\|\tau-\tau'\|(\|\tau\|+\|\tau'\|)$.
By \eqref{philip},
$|\phi_\gamma(x)-\phi_\gamma(x')|^{-1}
=(1+\theta_\gamma(x,x'))\,|x-x'|^{-1}$ with $\theta_\gamma\to0$ uniformly, so
with $w(x):=\Phi''(\rho_{\mu,\ep}(x))^{-1}\approx t(x)^2$ the integral in
\eqref{HSconv} equals
$$
\iint w(x)w(x')\,\frac{|b_\gamma(x)b_\gamma(x')(1+\theta_\gamma)-1|^2}
{|x-x'|^2}\,dx\,dx'.
$$
On $\{t(x)\ge\eta_\gamma,\ t(x')\ge\eta_\gamma\}$ the numerator factor is
$o(1)$ uniformly by \eqref{claimb}, and
$$
\sup_{x\in\bOmega}\int_{\bOmega}\frac{w(x')}{|x-x'|^2}\,dx'
\le\sup_{x}\int_{\bOmega}\frac{2t(x)^2+2|x-x'|^2}{|x-x'|^2}\,dx'<\infty,
$$
since $t(x')\le t(x)+|x-x'|$ and $|x-x'|^{-2}$ is integrable in three
dimensions; hence this region contributes $o(1)$. On the complementary region,
say $t(x)<\eta_\gamma$, \eqref{claimb} gives
$|b_\gamma(x)b_\gamma(x')(1+\theta_\gamma)-1|
\le C\big(t(x)t(x')\big)^{-|e(\gamma)|/2}$, so its contribution is at most
$$
C\int_{\{t(x)<\eta_\gamma\}}t(x)^{2-|e(\gamma)|}\,dx\;\cdot\;
\sup_{x}\int_{\bOmega}\frac{t(x')^{2-|e(\gamma)|}}{|x-x'|^2}\,dx'
=O\big(\eta_\gamma^{\,3-|e(\gamma)|}\big),
$$
and \eqref{HSconv} follows.

\emph{Part 2: the rotation term.} With $r_c(y):=\sqrt{y_1^2+y_2^2}$ and
$g_r^\gamma(y):=r_c(y)\mathbf 1_{\{r_c(y)<r\}}$, the definition \eqref{defT}
pulled back by $\mathcal U_\gamma$ reads
\begin{equation}
\label{That}
\ip{\widehat T^{\gamma}\tau,\tau}
=\int_0^{R^\gamma(0)}\frac{F^\gamma_\tau(r)^2}
{r\,\bar\rho^{\gamma}(r)}\,dr,\qquad
F^\gamma_\tau(r):=\frac1{2\pi}\ip{\tau,h^\gamma_r}_{X_1},\quad
h^\gamma_r:=\frac{b_\gamma\,(g_r^\gamma\circ\phi_\gamma)}
{\Phi''(\rho_{\mu,\ep})},\quad
\bar\rho^{\gamma}(r):=\int_{\R}\rho^{\gamma}_{\mu,\ep}(r,z)\,dz,
\end{equation}
where $F^\gamma_\tau(r)$ is the pulled-back cumulative marginal
$F^{\gamma}_{\sigma}(r)$ of \eqref{defF}; the limit form
$\ip{T_{\mu,\ep}\tau',\tau'}$ has the same shape with
$h_r:=g_r/\Phi''(\rho_{\mu,\ep})$, $g_r(x):=r_c(x)\mathbf1_{\{r_c(x)<r\}}$,
and $\bar\rho(r):=\int_{\R}\rho_{\mu,\ep}(r,z)\,dz$. We use the following
ingredients, all uniform in $r$ and in $\gamma$:
\begin{itemize}
\item[(a)] $\sup_{r}\|h^\gamma_r-h_r\|_{X_1}\to0$: indeed
$\|h^\gamma_r-h_r\|_{X_1}^2
=\int_{\bOmega}|b_\gamma\,g_r^\gamma\circ\phi_\gamma-g_r|^2
\Phi''(\rho_{\mu,\ep})^{-1}dx$; on $\{t\ge\eta_\gamma\}$,
$b_\gamma\to1$ uniformly by \eqref{claimb} and
$g_r^\gamma\circ\phi_\gamma\to g_r$ boundedly, pointwise except on the annulus
$\{x:|r_c(x)-r|\le C|\gamma-4/3|\}$ of volume $O(|\gamma-4/3|)$; on
$\{t<\eta_\gamma\}$ the integral is
$O\big(\int_{\{t<\eta_\gamma\}}t^{2-|e(\gamma)|}dx\big)
=O(\eta_\gamma^{3-|e(\gamma)|})$ by \eqref{claimb}.
\item[(b)] $\bar\rho^{\gamma}(r)\to\bar\rho(r)$ uniformly in $r$ (Lemma
\ref{lem:gammacont}(i)), $R^\gamma(0)\to R(0)$, $\bar\rho(r)\ge c(\eta)>0$
for $r\le R(0)-2\eta$, and, by \eqref{rhocompare} and the boundary
geometry near the equator,
\begin{equation}
\label{barrho}
\bar\rho(r)\approx(R(0)-r)^{7/2},\qquad
\bar\rho^{\gamma}(r)\approx\big(R^\gamma(0)-r\big)^{7/2-e(\gamma)}
\end{equation}
near the respective equatorial radii.
\item[(c)] Endpoint bounds: since both $\sigma$ and $\tau'$ are
mass-preserving,
$F^{\gamma}_{\tau}(R^\gamma(0))=(2\pi)^{-1}\int_{\bOmega^\gamma}\sigma\,dy=0$
and $F_{\tau'}(R(0))=(2\pi)^{-1}\int_{\bOmega}\tau'\,dx=0$; hence, writing
$F^{\gamma}_{\tau}(r)=-\int_r^{R^\gamma(0)}s\big(\int\sigma(s,z)\,dz\big)ds$
and using \eqref{rhocompare},
\begin{equation}
\label{Fbound}
|F^{\gamma}_{\tau}(r)|
\le C\|\tau\|_{X_1}\,\big(R^\gamma(0)-r\big)^{9/4-e(\gamma)/2},
\qquad
|F_{\tau'}(r)|\le C\|\tau'\|_{X_1}\,(R(0)-r)^{7/4},
\end{equation}
while near $r=0$ one has $|F^{\gamma}_{\tau}(r)|
\le C\|\tau\|_{X_1}r^{3/2}$, $|F_{\tau'}(r)|\le C\|\tau'\|_{X_1}r^{3/2}$ and
$\bar\rho^{\gamma}(0)\to\bar\rho(0)>0$.
\end{itemize}
Split the $r$-integrals at $r=\eta_\gamma$ and at distance $2\eta_\gamma$
from the respective endpoints. The endpoint layers contribute
$O\big(\eta_\gamma^{\,1-2|e(\gamma)|}\big)$ to
$\ip{\widehat T^{\gamma}\tau,\tau}$ and $O(\eta_\gamma)$ to
$\ip{T_{\mu,\ep}\tau',\tau'}$ (by \eqref{barrho}--\eqref{Fbound}), and the
origin layers $O(\eta_\gamma^3)$. On the main interval
$[\eta_\gamma,\min(R(0),R^\gamma(0))-2\eta_\gamma]$, the weights converge
uniformly by (b), and by (a),
$$
\sup_r|F^{\gamma}_{\tau}(r)-F_{\tau'}(r)|
\le C\|\tau-\tau'\|_{X_1}
+\|\tau\|_{X_1}\sup_r\|h^\gamma_r-h_r\|_{X_1}.
$$
Combining everything gives
$|\ip{\widehat T^{\gamma}\tau,\tau}-\ip{T_{\mu,\ep}\tau',\tau'}|
\le C\|\tau-\tau'\|_{X_1}(\|\tau\|_{X_1}+\|\tau'\|_{X_1})
+\omega(\gamma)(\|\tau\|_{X_1}^2+\|\tau'\|_{X_1}^2)$, and \eqref{Kconv}
follows.
\end{proof}

\begin{theorem}[Stability for $\gamma$ near $4/3$; Theorem \ref{thm:main}]
\label{thm:gamstab}
Fix $\mu\in[\mu_0,\mu_1]$ and $\ep\in(0,\ep_0)$. Then there exists
$\ep_1=\ep_1(\mu_0,\mu_1,\ep)>0$ such that for every
$\gamma\in(4/3-\ep_1,4/3]$, the rotating polytropic star solution
$(\rho^{\gamma}_{\mu,\ep},\ep r\mathbf{e}_\theta)$ of \eqref{SteadyGamma} is
\emph{spectrally stable} with respect to axi-symmetric perturbations: the
corresponding linearized operator
$\mathbf{J}^{\gamma}_{\mu,\ep}\mathbf{L}^{\gamma}_{\mu,\ep}$ (defined as in
\eqref{defJL} with $\rho^{\gamma}_{\mu,\ep}$ in place of $\rho_{\mu,\ep}$) satisfies
$$
\sigma\big(\mathbf{J}^{\gamma}_{\mu,\ep}\mathbf{L}^{\gamma}_{\mu,\ep}\big)\subset i\R
\qquad\text{and}\qquad
\left|e^{t\mathbf{J}^{\gamma}_{\mu,\ep}\mathbf{L}^{\gamma}_{\mu,\ep}}\right|
\le C(1+|t|)^3\quad\forall t\in\R,
$$
for some $C>0$ independent of $\gamma\in(4/3-\ep_1,4/3]$, $\mu\in[\mu_0,\mu_1]$.
\end{theorem}

\begin{proof}
Let $L^{\gamma}_{\mu,\ep}$, $T^{\gamma}_{\mu,\ep}$, $K^{\gamma}_{\mu,\ep}$ be
the operators of \eqref{defL}, \eqref{defT}, \eqref{defK2} with
$\rho^{\gamma}_{\mu,\ep}$ in place of $\rho_{\mu,\ep}$ (in particular
$K^{\gamma}_{\mu,\ep}=L^{\gamma}_{\mu,\ep}+8\pi\ep^2T^{\gamma}_{\mu,\ep}$ and
$\Upsilon\equiv4\ep^2>0$), and let $\mathcal U_\gamma$ be the transport of
Lemma \ref{lem:transport}.

\emph{Step 1 (convergence of the constraint subspaces).} Set
$\widetilde R^{\gamma}:=\{\delta\rho\in R(B_1^{\gamma})\,:\,
\ip{\delta\rho,\p_z\rho^{\gamma}_{\mu,\ep}}_{X^{\gamma}_1}=0\}$. By Lemma
\ref{lem:transport}(iv), the pulled-back constraint subspaces
\begin{equation}
\label{defRhat}
\widehat{\mathcal R}^{\gamma}:=\mathcal U_\gamma\widetilde R^{\gamma}
=\{\tau\in X^{\mu,\ep}_1\,:\,\ip{\tau,\ell_\gamma}_{X_1}=0,\
\ip{\tau,p_\gamma}_{X_1}=0\}
\end{equation}
converge to
$\widetilde R^{\mu,\ep}=\{\tau\in X^{\mu,\ep}_1\,:\,\ip{\tau,\ell}_{X_1}=0,
\ip{\tau,p}_{X_1}=0\}$ in the gap topology, since $\ell_\gamma\to\ell$ and
$p_\gamma\to p$ in $X^{\mu,\ep}_1$. Hence every
$\tau\in\widehat{\mathcal R}^{\gamma}$ admits a projection
$\tau'\in\widetilde R^{\mu,\ep}$ with
\begin{equation}
\label{projest}
\|\tau-\tau'\|_{X^{\mu,\ep}_1}\le c(\gamma)\longrightarrow0
\qquad\text{uniformly in }\tau.
\end{equation}

\emph{Step 2 (uniform positivity away from the translation mode).} Let
$\sigma\in\widetilde R^{\gamma}$ with $\|\sigma\|_{X^{\gamma}_1}=1$, set
$\tau:=\mathcal U_\gamma\sigma$, and choose
$\tau'\in\widetilde R^{\mu,\ep}$ as in Step 1. Since $\sigma$ and $\tau'$ are
mass-preserving, Proposition \ref{prop:Kconv} applies; together with Theorem
\ref{thm:strict}(ii) this gives
$$
\ip{K^{\gamma}_{\mu,\ep}\sigma,\sigma}
\ge\ip{K_{\mu,\ep}\tau',\tau'}-C\|\tau-\tau'\|_{X_1}-\omega(\gamma)
\ge\delta'\|\tau'\|_{X_1}^2-o(1)\ge\frac{\delta'}{2},
$$
for all $\gamma$ sufficiently close to $4/3$, i.e.
\begin{equation}
\label{strictposgamma}
\ip{K^{\gamma}_{\mu,\ep}\delta\rho,\delta\rho}
\ge\frac{\delta'}{2}\,\|\delta\rho\|_{X^{\gamma}_1}^2
\qquad\forall\,\delta\rho\in\widetilde R^{\gamma}.
\end{equation}

\emph{Step 3 (nonnegativity on the mass-preserving subspace).} For $\gamma$ close
to $4/3$, the translation mode $\partial_z\rho^{\gamma}_{\mu,\ep}$ is an even-in-$z$
odd function with $\int\partial_z\rho^{\gamma}_{\mu,\ep}\,dx=0$, so
$\partial_z\rho^{\gamma}_{\mu,\ep}\in R(B_1^{\gamma})$, and
$L^{\gamma}_{\mu,\ep}\partial_z\rho^{\gamma}_{\mu,\ep}=0$ by the $z$-symmetry, hence
$K^{\gamma}_{\mu,\ep}\partial_z\rho^{\gamma}_{\mu,\ep}=0$ (as
$K^{\gamma}_{\mu,\ep}=L^{\gamma}_{\mu,\ep}$ on odd functions). Every
$\delta\rho\in R(B_1^{\gamma})$ decomposes as
$\delta\rho=b\,\partial_z\rho^{\gamma}_{\mu,\ep}+\delta\rho'$ with
$\delta\rho'\in\widetilde R^{\gamma}$, and by self-duality
$\ip{K^{\gamma}_{\mu,\ep}\partial_z\rho^{\gamma}_{\mu,\ep},\delta\rho'}=0$, so
$$
\ip{K^{\gamma}_{\mu,\ep}\delta\rho,\delta\rho}
=\ip{K^{\gamma}_{\mu,\ep}\delta\rho',\delta\rho'}
\ge\frac{\delta'}{2}\|\delta\rho'\|_{X^{\gamma}_1}^2\ge 0.
$$
Thus $K^{\gamma}_{\mu,\ep}|_{R(B_1^{\gamma})}\ge 0$, i.e.
$n^-(K^{\gamma}_{\mu,\ep}|_{R(B_1^{\gamma})})=0$.

\emph{Step 4 (spectral stability).} For $|\gamma-4/3|$ small, the hypotheses of
Theorem \ref{thm:criterion} are satisfied by the admissible solutions of
Lemma \ref{lem:gammacont} (the boundary $\partial\bOmega^{\gamma}_{\mu,\ep}$ is
$C^2$ with positive curvature near the equator and
$\rho^{\gamma}_{\mu,\ep}\approx\operatorname{dist}(\cdot,\partial
\bOmega^{\gamma}_{\mu,\ep})^{1/(\gamma-1)}$ by admissibility, and
$\Upsilon\equiv4\ep^2>0$). Hence, by Theorem \ref{thm:criterion} and Step 3,
$$
\dim E^u=\dim E^s=n^-(K^{\gamma}_{\mu,\ep}|_{R(B_1^{\gamma})})=0,
$$
so $\mathbf{X}^{\gamma}=E^c$, $\sigma(\mathbf{J}^{\gamma}\mathbf{L}^{\gamma})
\subset i\R$, and $\left|e^{t\mathbf{J}^{\gamma}\mathbf{L}^{\gamma}}\right|
\le M(1+|t|)^3$ for all $t\in\R$. Here $M$ is uniform in
$\gamma\in(4/3-\ep_1,4/3]$, $\mu\in[\mu_0,\mu_1]$: the structural constants
entering Theorem \ref{thm:criterion} (the admissibility bounds of Definition
\ref{def:admissible}, the $C^2$ and curvature bounds of
$\p\bOmega^{\gamma}_{\mu,\ep}$, and the physical vacuum slope) are uniform in
$\gamma$ by the $C^1$ convergence of the extended enthalpies in Lemma
\ref{lem:gammacont}, and the spectral gap is uniform by Proposition
\ref{prop:decomp} and \eqref{strictposgamma}.
\end{proof}

\begin{remark}[Relation with the instability for fixed $\gamma<4/3$]
\label{rem:limits}
The order of the limits in Theorem \ref{thm:gamstab} is essential: we fix the small
angular velocity $\ep>0$ and let $\gamma\uparrow4/3$. By \cite[Theorem 3.2]{LW2023},
if instead one fixes $\gamma<4/3$ and lets $\ep\to0$, the rotating stars are
unstable, since the non-rotating polytropic stars are unstable for $\gamma<4/3$.
The two limits do not commute: for each fixed $\gamma<4/3$ the instability persists
only for $\ep<\ep_0(\gamma)$, with $\ep_0(\gamma)\to0$ as $\gamma\uparrow4/3$.
Equivalently, for the fixed small $\ep>0$ of Theorem \ref{thm:gamstab}, the
instability mechanism of order $|\gamma-4/3|$ (which vanishes at $\gamma=4/3$) is
dominated by the stabilizing rotational term of order $\ep^2$ when
$|\gamma-4/3|$ is smaller than $\ep^2$ up to a constant. Theorem
\ref{thm:gamstab} makes this rigorous through the uniform convergence
\eqref{Kconv} and the strict positivity \eqref{strictposgamma}.
\end{remark}

\section{Discussion}
\label{sec:discussion}

\subsection{The cases $\gamma\ne 4/3$}
\label{sec:gammaneq}

The critical exponent $\gamma=4/3$ is the borderline for the stability of polytropic
stars, and rotation does not change the picture away from the critical value.

\begin{itemize}
\item \emph{$\gamma>4/3$:} the non-rotating stars are stable. By
\cite[Theorem 3.3]{LW2023}, the slowly rotating stars
$(\rho_{\mu,\ep},\ep r\mathbf{e}_\theta)$ are spectrally stable for
$\ep$ small. 

\item \emph{$\gamma<4/3$:} the non-rotating stars are unstable. By
\cite[Theorem 3.2]{LW2023}, for any Rayleigh stable angular velocity profile the
rotating stars $(\rho_{\mu,\ep},\ep r\mathbf{e}_\theta)$ remain unstable for $\ep>0$
small (the instability of the non-rotating star is stable under the small
perturbation of rotation). Hence rotation \emph{cannot} stabilize the subcritical
polytropes. In particular, Theorem \ref{thm:main} is sharp in the following sense:
for fixed $\gamma<4/3$ and $\ep\to0$ rotation cannot stabilize, whereas for fixed
$\ep>0$ and $\gamma\uparrow4/3$ it does (Theorem \ref{thm:main}, cf. Remark
\ref{rem:limits}); $\gamma=4/3$ is exactly the critical value where the (marginal)
stability of the non-rotating star is converted into strict stability by rotation.
\end{itemize}

\subsection{Relation to the uniqueness theorem of \cite{W2024}}
\label{sec:uniqueness}

In \cite{W2024} (see also \cite{JM2019}) it was proved that for the uniformly rotating
supermassive stars ($\gamma=4/3$) with fixed small angular velocity $\ep>0$, the total
mass satisfies
$$
\partial_\mu M_{\mu,\ep}<0,\qquad \partial_{\ep^2}M_{\mu,\ep}>0,
\qquad M_{\mu,\ep}>M_c:=M_{\mu,0}
$$
for $\ep$ small, where $M_c$ is the Chandrasekhar limit. In particular, for every pair
$(M,\ep)$ with $\ep>0$ small and $M$ in a right neighborhood of $M_c$, there exists a
\emph{unique} central density $\mu=\mu(M,\ep)$ such that
$M_{\mu(M,\ep),\ep}=M$. Thus the family of slowly uniformly rotating supermassive
stars is well-parametrized by the pair $(M,\ep)$, and by Theorem \ref{thm:main}
(with $\gamma=4/3$) every
member of this family is spectrally stable. The combination of the uniqueness theorem
of \cite{W2024} and the spectral stability of Theorem \ref{thm:main} provides a
complete dynamical picture of slowly uniformly rotating supermassive stars at the
linear level.

Note that $\partial_\mu M_{\mu,\ep}<0$ for the family with fixed $\ep$ means that the
turning point principle \emph{fails} for this family: stability holds although the
mass is strictly decreasing in the central density. This is consistent with, and in
fact an extreme manifestation of, the phenomenon discovered in \cite[Theorem 3.4]
{LW2023} for $\gamma>4/3$.

\subsection{The case of fixed angular momentum distribution}
\label{sec:fixedj}

It is instructive to compare with the family of slowly rotating stars with a
\emph{fixed angular momentum distribution} $j=j(m)$ (where $m=m_{\rho}(r)=
2\pi\int_0^r s\int_{\R}\rho(s,z)\,dz\,ds$ is the mass in the cylinder), studied in
\cite[Section 3.2]{LW2023}. For such a family, the steady equation is
\begin{equation}
\label{SteadyJ}
-\ep^2\int_0^r\frac{j^2(m_{\rho_{\mu,\ep}}(s))}{s^3}\,ds
+\Phi'(\rho_{\mu,\ep})+V_{\mu,\ep}+c_{\mu,\ep}=0,
\end{equation}
and the reduced operator takes the form
\begin{equation}
\label{defKJ}
\ip{K^J_{\mu,\ep}\delta\rho,\delta\rho}
=\ip{L_{\mu,\ep}\delta\rho,\delta\rho}
+2\pi\ep^2\int_0^{R_{\mu,\ep}}\frac{\partial_m(j^2)(m_{\rho_{\mu,\ep}}(r))}{r^3}
\left(\int_0^r s\int_{\R}\delta\rho(s,z)\,dz\,ds\right)^2 dr.
\end{equation}
For the fixed angular momentum family, the reduced operator is
$K^J_{\mu,\ep}=L_{\mu,\ep}+2\pi\ep^2 T^J_{\mu,\ep}$, where $T^J_{\mu,\ep}\ge 0$ is the
positive form
$\ip{T^J_{\mu,\ep}\delta\rho,\delta\rho}=\int_0^{R_{\mu,\ep}}
\frac{\partial_m(j^2)(m_{\rho_{\mu,\ep}}(r))}{r^3}F_{\delta\rho}(r)^2\,dr$.
If $j$ is strictly increasing (equivalently, the corresponding angular velocity is
Rayleigh stable, $\partial_m(j^2)>0$), then $K^J_{\mu,\ep}\ge L_{\mu,\ep}$.
For the uniform rotation profile, $j(m_{\rho_{\mu,\ep}}(r))=r$, one has
$\partial_m(j^2)(m)=\frac{1}{\pi\int_{\R}\rho_{\mu,\ep}(r,z)\,dz}$, so that
$2\pi\ep^2T^J_{\mu,\ep}=2\ep^2\int_0^{R_{\mu,\ep}}
\frac{F_{\delta\rho}(r)^2}{r^3\int_{\R}\rho_{\mu,\ep}(r,z)\,dz}\,dr$, which is the
fixed-angular-momentum version of the rotation term $8\pi\ep^2T_{\mu,\ep}$ in
\eqref{defK2} (the two forms agree on the subspace where the corresponding
reductions are equivalent; see also Remark \ref{rem:choice62}).
For the fixed angular momentum family one can construct a modified operator
$G^J_{\mu,\ep}$ (e.g. $G^J_{\mu,\ep}=L_{\mu,\ep}+\frac32\ep^2\pi T^J_{\mu,\ep}$)
such that $f_{\mu,\ep}$ is again a reduced kernel direction and
$K^J_{\mu,\ep}-G^J_{\mu,\ep}\ge 0$, so the same argument as in Section
\ref{sec:reduced} proves the spectral stability at $\gamma=4/3$ without any
condition on $\partial_\mu M_{\mu,\ep}$. We work in this paper directly with the
fixed-angular-velocity formulation \eqref{Steady}, which is the natural one for
Theorem \ref{thm:main} and for the uniqueness theory of \cite{W2024}.

\subsection{Open problems}
\label{sec:open}

We close with some remarks on possible extensions.
\begin{itemize}
\item \emph{Nonlinear stability.} The linear (spectral) stability established in
Theorem \ref{thm:main}, together with the polynomial growth bound
\eqref{estimate-center}, is the first step toward a nonlinear stability theory.
For non-rotating stars, nonlinear stability in the polytropic case was obtained by
\cite{LS2009,LS2011} via variational methods; the extension to slowly rotating
supermassive stars requires a control of the Casimir invariants (generalized angular
momentum, cf. \cite[Section 2.3]{LW2023}) which is beyond the scope of the present
paper.
\item \emph{The dependence on the angular velocity profile.} Theorem \ref{thm:main}
holds for uniform rotation, for which $\Upsilon\equiv 4\ep^2>0$. For a general
Rayleigh stable profile $\omega_0(r)$, the same argument applies verbatim as long as
$\Upsilon(r)>0$; the positivity of the Rayleigh discriminant is exactly the
mechanism that enforces $K_{\mu,\ep}\ge L_{\mu,\ep}$ and hence the stability.
\item \emph{Rapid rotation.} For rapidly rotating supermassive stars (whose
existence for $\gamma=4/3$ was established by variational methods in
\cite{LS2008}), the stability problem is open; the rigid rotation may reach
bifurcation points (e.g. the Maclaurin-like sequences) where the linearized operator
develops additional neutral modes.
\end{itemize}

\section*{Acknowledgements}
Yucong Wang was supported in part by the National
Natural Science Foundation of China (No.~12501290), the Furong Young Talents
Program for Science and Technology Innovation of Hunan Province
(No.~2026RC3168), the Natural Science Foundation of Hunan Province
(No.~2025JJ60068), the Scientific Research Fund of the Hunan Provincial
Education Department (No.~25B0150), the 111 Project (No.~D23017), and the
Program for Science and Technology Innovative Research Teams in Higher
Educational Institutions of Hunan Province, China.

Moonshot AI's Kimi-K3 was used during revision to assist with language editing,
\LaTeX{} reorganization, literature checks, and mathematical consistency
checks.  The authors assume responsibility for all content.

\bibliographystyle{plain}

\end{document}